\documentclass{amsart}

\usepackage[latin1]{inputenc}
\usepackage{amsmath, amsthm, amssymb, booktabs, multirow, graphicx,
  booktabs, bm, float}
\usepackage{mathtools}
\usepackage{enumitem}
\usepackage{dialogue}
\usepackage[stable]{footmisc}
\usepackage{tikz,nicefrac}
\usetikzlibrary{calc,matrix,arrows,decorations.markings}
\usepackage[colorlinks]{hyperref}
\usetikzlibrary{
  graphs,
  graphs.standard
}

\newtheorem{defin}{Definition}[section]

\newtheorem{theorem}[defin]{Theorem}
\newtheorem{proposition}[defin]{Proposition}
\newtheorem{lemma}[defin]{Lemma}
\newtheorem{corollary}[defin]{Corollary}

\newtheorem{remark}[defin]{Remark}

\newcommand{\C}{\mathbb{C}}
\newcommand{\R}{\mathbb{R}}

\newcommand{\Z}{\mathbb{Z}}

\newcommand{\N}{\mathbb{N}}

\newcommand{\Tcal}{\mathcal{T}}

\makeatletter
\newcommand{\bigperp}{%
  \mathop{\mathpalette\bigp@rp\relax}%
  \displaylimits
}

\newcommand{\bigp@rp}[2]{%
  \vcenter{
    \m@th\hbox{\scalebox{\ifx#1\displaystyle2.1\else1.5\fi}{$#1\perp$}}
  }%
}
\makeatother

\DeclareMathOperator{\Tr}{Tr}

\DeclareMathOperator{\Harm}{Harm}
\DeclareMathOperator{\lspan}{span}
\DeclareMathOperator{\Sym}{Sym}

\definecolor{arne}{rgb}{0.2,.7,0.9}

\definecolor{aurelio}{rgb}{0.9,.2,0.9}

\definecolor{frank}{rgb}{0.8,0.2,0.5}

\definecolor{marc}{rgb}{1,0.65,0}

\begin{document}

\title{Critical even unimodular lattices in the Gaussian core model}

\address{A.~Heimendahl, A.~Marafioti, A.~Thiemeyer, F.~Vallentin,
  M.C.~Zimmermann, Department Mathematik/Informatik, Abteilung
  Mathematik, Universit\"at zu K\"oln, Weyertal~86--90, 50931 K\"oln,
  Germany}

\author{Arne Heimendahl}

\email{arne.heimendahl@uni-koeln.de}

\author{Aurelio Marafioti}
\email{aureliomarafioti@gmail.com}

\author{Antonia Thiemeyer}
\email{athiemey@smail.uni-koeln.de}

\author{Frank Vallentin}
\email{frank.vallentin@uni-koeln.de}

\author{Marc Christian Zimmermann}
\email{marc.christian.zimmermann@gmail.com}

\date{May 24, 2022}


\begin{abstract}
  We consider even unimodular lattices which are critical for
  potential energy with respect to Gaussian potential functions in the
  manifold of lattices having point density $1$. All even unimodular
  lattices up to dimension $24$ are critical. We show how to determine
  the Morse index in these cases. While all these lattices are either local
  minima or saddle points, we find lattices in dimension $32$ which
  are local maxima. Also starting from dimension $32$ there are
  non-critical even unimodular lattices.
\end{abstract}

\maketitle

\markboth{A. Heimendahl, A. Marafioti, A. Thiemeyer, F. Vallentin, and
  M.C. Zimmermann}{Critical even unimodular lattices in the Gaussian
  core model}

\section{Introduction}

Let $L \subseteq \mathbb{R}^n$ be an $n$-dimensional lattice (a
discrete subgroup of $\R^n$ of full rank). Let $f : (0, \infty) \to \R$
be a nonnegative function, then the $f$-potential energy of $L$ is
defined as
\[
  \mathcal{E}(f,L) = \sum_{x \in L \setminus\{0\}} f(\|x\|^2).
\]

In this paper we are mainly interested in Gaussian potential functions
$f_\alpha(r) = e^{-\alpha r}$ with $\alpha > 0$. Point configurations
which interact via such a Gaussian potential function are referred to
as the Gaussian core model. They are natural physical systems (see
\cite{Stillinger1976a}) and they are mathematically quite general. By
Bernstein's theorem (see \cite[Theorem 12b, page 161]{Widder1941a}),
Gaussian potential functions span the convex cone of completely
monotonic functions ($C^\infty$-functions $f$ with
$(-1)^k f^{(k)}\geq 0$ for all $k \in \N$) of squared Euclidean
distance.

We are interested in a local analysis of the function
$L \mapsto \mathcal{E}(f_\alpha, L)$ when $L$ varies in the manifold of
rank $n$ lattices having point density $1$, which means that the
number of lattice points per unit volume equals $1$.  In particular,
we want to understand which even unimodular lattices are critical
points in the Gaussian core model and which type they have.

Recall that a lattice $L$ is called unimodular if it coincides with its dual
lattice, which is defined as
\[
  L^* = \{y \in \R^n : x \cdot y \in \Z \text{ for all } x \in L\},
\]
where $x \cdot y$ denotes the standard inner product of
$x, y \in \R^n$. The lattice $L$ is called even if for every lattice
vector $x \in L$ the inner product $x \cdot x$ is an even integer. It
is well-known that in a given dimension the number of even unimodular
lattices is finite and that they exist only in dimensions which are
divisible by $8$. Furthermore, dimensions $8$ and $24$ seem to be very
special. Cohn, Kumar, Miller, Radchenko and Viazovska \cite{Cohn2019a}
proved that the $E_8$ root lattice in dimension $8$ and the Leech
lattice $\Lambda_{24}$ in dimension $24$ are universally optimal point
configurations in their dimensions. This means that they minimize
$f$-potential energy for all point configurations having density $1$
in their dimensions (not only for lattices) and for all completely
monotonic functions of squared Euclidean distance.

\subsection{Structure of the paper and main results}

In Section~\ref{sec:Results} we present our concrete results. Here we
summarize the phenomena which occur.

\subsubsection{Dimension 8}

Section~\ref{ssec:Dimension-8}: In dimension $8$ the $E_8$ root
lattice is the only even unimodular lattice in dimension $8$ as
observed by Mordell \cite{Mordell1938a}. It is universally optimal. In
particular, it is a local minimum for $f_\alpha$-potential
energy. This was first proved by Sarnak and Str\"ombergsson
\cite{Sarnak2006a}, see also Coulangeon \cite{Coulangeon2006a}.

\subsubsection{Dimension 16}

Section~\ref{ssec:Dimension-16}: In dimension $16$ there are two even
unimodular lattices $D_{16}^+$ and $E_8 \perp E_8$, first classified
by Witt \cite{Witt1941a}. Both of them are critical and we show that
$D_{16}^+$ is a local minimum for $f_\alpha$-potential energy whenever
$\alpha$ is large enough and that $E_8 \perp E_8$ is a saddle point
whenever $\alpha$ is large enough. Our numerical computations strongly
suggest that $E_8 \perp E_8$ is a saddle point for all values of
$\alpha$.

\subsubsection{Dimension 24}

Section~\ref{ssec:Dimension-24}: Apart from the universally optimal
Leech lattice there are $23$ further even unimodular lattices in
dimension $24$. They were first classified by Niemeier
\cite{Niemeier1973a}. Again they are all critical. We show how to
determine their Morse index. We always find either local minima or
saddle points.

\subsubsection{Dimension 32}

Section~\ref{ssec:Dimension-32}: It is known that there are more than
80 millions even unimodular lattices in dimension $32$; cf.\ Serre
\cite{Serre1973a}. A complete classification has not been achieved
yet. We show that not all of them are critical. We also show that
there exist local \textit{maxima} for $f_\alpha$-potential energy. This
existence of local maxima answers a question of Regev and
Stephens-Davidowitz \cite{Regev2016a} which arose in their proof
strategy of the reverse Minkowski theorem; see also the exposition
\cite{Bost2018a} by Bost for a broad perspective. A similar
phenomenon, a local maximum for the covering density of a lattice, was
earlier found by Dutour Sikiri\'c, Sch\"urmann, and
Vallentin~\cite{Dutour2012a}.

\subsubsection{Proof techniques}

To prove these results we make use of the theory of lattices and
codes, especially spherical designs, theta series with spherical
coefficients, and root systems. We recall these tools in
Section~\ref{sec:Toolbox}. In Section~\ref{sec:Strategy} we describe
our strategy which is based on the explicit computation of the
signature of the Hessian of the function
$L \mapsto \mathcal{E}(f_\alpha, L)$. To work out this strategy it is
necessary to explicitly compute the eigenvalues of a symmetric matrix
which is parametrized by root systems. This is done in
Section~\ref{sec:Eigenvalues}.

\section{Toolbox}
\label{sec:Toolbox}

In this section we introduce the tools we shall apply later in this
paper. For more information we refer to the standard literature on
lattices and codes, in particular to Conway and Sloane
\cite{Conway1988a}, Ebeling \cite{Ebeling1994a}, Serre
\cite{Serre1973a}, Venkov \cite{Venkov2001a}, Nebe
\cite{Nebe2013a}. Readers familiar with lattices and codes might like
to skip immediately to the next section.

\subsection{Spherical designs}

A finite set $X$ on the sphere of radius $r$ in $\R^n$ denoted by
$S^{n-1}(r)$ is called a spherical $t$-design if
\[
  \int_{S^{n-1}(r)} p(x) \, dx = \frac{1}{|X|} \sum_{x \in X} p(x)
\]
holds for every polynomial $p$ of degree up to $t$. Here we integrate
with respect to the rotationally invariant probability measure on
$S^{n-1}(r)$.

If $X$ forms a spherical $2$-design, then
\begin{equation}
\label{eq:spherical-2-design}
  \sum_{x \in X} xx^{\sf T} = \frac{r^2 |X|}{n} I_n,
\end{equation}
holds, where $I_n$ denotes the identity matrix with $n$ rows/columns.

A polynomial $p \in \R[x_1, \ldots, x_n]$ is called harmonic if it
vanishes under the Laplace operator
\[
  \Delta p = \sum_{i=1}^n \frac{\partial^2 p}{\partial x_i^2} = 0.
\]
We denote the space of homogeneous harmonic polynomials of degree $k$
by $\Harm_k$. One can uniquely decompose every homogeneous polynomial
$p$ of even degree $k$
\begin{equation}
\label{eq:harmonic-decomposition}
p(x) = p_k(x) + \|x\|^2 p_{k-2}(x) + \|x\|^4 p_{k-4}(x) + \cdots +
\|x\|^k p_0(x)
\end{equation}
with $p_d \in \Harm_d$ and $d = 0, 2, \ldots, k$.

We can characterize that $X$ is a spherical $t$-design by saying that
the sum $\sum_{x \in X} p(x)$ vanishes for all homogeneous harmonic
polynomials $p$ of degree $1, \ldots, t$.

\smallskip

In the following we shall need the following technical lemma.

\begin{lemma}
\label{lemma:polynomial-p-H}
Let $H$ be a symmetric matrix with trace zero. The homogeneous polynomial
\[
p_H(x) = (x^{\sf T} H x)^2 = H[x]^2
\]
of degree four decomposes as in \eqref{eq:harmonic-decomposition}
\[
  p_H(x) = p_{H,4}(x) + \|x\|^2 p_{H,2}(x) + \|x\|^4 p_{H,0}(x)
\]
with $p_{H,d} \in \Harm_d$ and
\[
  p_{H,4}(x) = p_H(x) - \|x\|^2 \frac{4}{4+n} H^2[x] + \|x\|^4 \frac{2}{(4+n)(2+n)}
  \Tr H^2
\]
and
\[
  p_{H,0}(x) = \frac{2}{(2+n)n} \Tr H^2.
\]
\end{lemma}

\begin{proof}
  As a consequence of Euler's formula we have for a general harmonic
  polynomial $q \in \Harm_d$
\[
  \Delta \|x\|^2 q = (4d+2n)q + \|x\|^2 \Delta q = (4d+2n)q,
\]
and inductively
\begin{equation}
\label{eq:Euler}
  \Delta \|x\|^{2(k+1)} q = (k+1) (4k+4d+2n)\|x\|^{2k} q,
\end{equation}
see for example \cite[Lemma 3.5.3]{Simon2015a}\footnote{The factor $2$
  in (3.5.11) is wrong in \cite{Simon2015a}; it should be $1$.}.

Using~\eqref{eq:harmonic-decomposition} we get
\[
  \begin{split}
  \Delta p_H \;  = \; & \Delta p_{H,4} + (8 + 2n) p_{H,2} + \|x\|^2 \Delta
  p_{H,2} + \Delta \|x\|^4 p_{H,0}\\
  = \; & (8 + 2n) p_{H,2} + 2(4+2n) \|x\|^2 p_{H,0}.
  \end{split}
\]
Applying the Laplace operator another time yields
\[
  \Delta^2 p_H = 8 n (n+2) p_{H,0}.
\]
On the other hand, one can compute $\Delta^2 p_H$ directly. We have
\[
  H[x] = \sum_{i=1}^n \sum_{j = 1}^n H_{ij} x_i x_j
\]
and therefore
\[
  \Delta H[x] = 2 \sum_{i=1}^n H_{ii} = 2 \Tr H.
\]
Using the product formula for the Laplace operator and the symmetry of
$H$ we get
\[
  \Delta p_H = \Delta H[x]^2 = 2(H[x] \Delta H[x] +  \nabla
H[x] \cdot \nabla H[x])  = 4 (\Tr H) H[x] + 8 H^2[x].
\]
Therefore
\[
  \Delta^2 p_H = 8 (\Tr H)^2 + 16 \Tr H^2
\]
and so
\[
  p_{H,0} = \frac{2}{n(n+2)} \Tr H^2,
\]
where the last equation follows from $\Tr H = 0$.

We already computed
\[
 \Delta \|x\|^4 p_{H,0} = 2(4+2n) \|x\|^2 p_{H,0} = \frac{8}{n} \Tr H^2 \|x\|^2.
\]

Now we determine $p_{H,2}$ when $\Tr H = 0$:
\[
(8+2n) p_{H,2} = \Delta p_H - \|x\|^2 \frac{8}{n} \Tr H^2  =  8 H^2[x] -
\|x\|^2 \frac{8}{n} \Tr H^2.
\]

Finally we get $p_{H,4}$:
\[
  p_{H,4} = p_H - \|x\|^2 \frac{4}{4+n} H[x]^2 + \|x\|^4 \frac{2
    }{(4+n)(2+n)} \Tr H^2. \qedhere
\]
\end{proof}

\subsection{Theta series with spherical coefficients}

We will make use of theta series with spherical coefficients. Let
$L \subseteq \R^n$ be an even unimodular lattice and let
$p$ be a harmonic polynomial (sometimes also
called spherical polynomial).

We define the theta series of $L$ with spherical coefficients given by
$p$ by
\[
\Theta_{L,p}(\tau) = \sum_{x \in L} p(x) e^{\pi i \tau \|x\|^2} = \sum_{x \in L} p(x) q^{\frac{1}{2}\|x\|^2},
\]
where $\tau$ lies in the upper half plane $\{z \in \C : \Im(z) > 0\}$
and where $q = e^{2 \pi i \tau}$.

If $p = 1$ we also write $\Theta_{L}$
instead of $\Theta_{L,p}$. For $r \geq 0$ we define
\[
  L(r^2) = \{x \in L : x \cdot x = r^2\}.
\]
The set $L(r^2)$ is called a shell of $L$ if it is
not empty. Then
\[
  \Theta_L(\tau) = \sum_{m=0}^\infty a_m q^m \quad \text{with} \quad
  a_m =|L(2m)|.
\]
The theta series of $L$ is
related to its $f_\alpha$-potential energy through
\[
  \mathcal{E}(f_\alpha,L) = \Theta_L(\alpha i / \pi) - 1.
\]
Using the Poisson summation formula one sees that
\[
\Theta_L(iy) = y^{-n/2} \Theta_{L^*}(i/y) \; \text{ for } \; y > 0.
\]
In particular, when $L = L^*$ it is sufficient to consider Gaussian
potentials with $\alpha \geq \pi$.

If $p$ is a homogeneous harmonic polynomial
of degree $k$, then $\Theta_{L,p}$ is a modular form (for the full
modular group $\mathrm{SL}_2(\Z)$) of weight $n/2+k$. When $k > 1$
then $\Theta_{L,p}$ is a cusp form.  We only need that modular forms
form a graded ring which is isomorphic to the polynomial ring
$\C[E_4, E_6]$ in the (normalized) Eisenstein series
\[
  E_4(\tau) = 1 + 240 q + 2160 q^2 + 6720 q^3 + \cdots,
\]
and
\[
  E_6(\tau) = 1 - 504 q - 16632 q^2  - 122976 q^3 - \cdots,
\]
where the weight of the monomial $E_4^\alpha E_6^\beta$ is
$4\alpha + 6 \beta$. Generally, the normalized Eisenstein series are
given by
\[
  E_k(\tau) = 1 - \frac{2k}{B_k} \sum_{m = 1}^\infty \sigma_{k-1}(m)
  q^m \quad \text{for } k \geq 4,
\]
where $B_k$ is the $k$-th Bernoulli number and where
$\sigma_{k-1}(m) = \sum_{d|m} d^{k-1}$ is the sum of the $(k-1)$-th
powers of positive divisors of $m$. The space of cusp forms is a
principal ideal of the polynomial ring $\C[E_4, E_6]$ generated by the
modular discriminant
\[
  \Delta = \frac{1}{1728} (E_4^3 - E_6^2) = 0 + q - 24q^2 +
  252 q^3 \pm \cdots,
\]
which has weight $12$.

It is a standard fact that the cardinality $a_m = |L(2m)|$ of the shells
is asymptotically bounded, when $m$ tends to infinity, by
\[
  a_m = -\frac{n}{B_{n/2}} \sigma_{n/2-1}(m) + O(m^{n/4}),
\]
but in this paper we shall need a bound with explicit constants.

For this we we will use the following explicit bound by Jenkins and
Rouse \cite{Jenkins2011a} which relies on Deligne's proof of the Weil
conjectures: Let $f(\tau) = \sum_{m=1}^\infty a_m q^m$ be a cusp form
of weight $k$, let $\ell$ be the dimension of the space of cusp forms
of weight $k$, then
\begin{equation}
\begin{split}
  \label{eq:explicit-bound}
|a_m| \leq & \sqrt{\log(k)} \left(11 \cdot \sqrt{\sum_{r=1}^{\ell}
  \frac{|a_r|^{2}}{r^{k-1}}}
  + \frac{e^{18.72} (41.41)^{k/2}}{k^{(k-1)/2}} \cdot \left|\sum_{r=1}^{\ell} a_r e^{-7.288r} \right| \right)
  \\
  & \qquad  \cdot \;
d(m) m^{\frac{k-1}{2}},
\end{split}
\end{equation}
where $d(m)$ is the number of divisors of $m$.  

\smallskip

The following simple estimate will be helpful several times.

\begin{lemma}
 \label{lem:integral-test-estimate}
For $j \geq k/(2\alpha)$ we have
\begin{equation}
\label{eq:integral-test-estimate}
\sum_{m = j}^\infty m^k e^{-2\alpha m} \leq j^k e^{-2\alpha j} +
(2\alpha)^{-(k+1)} \Gamma(k+1, 2\alpha j),
\end{equation}
where
\[
  \Gamma(s,x) = \int_x^\infty t^{s-1} e^{-t} dt
\]
is the incomplete gamma function.
\end{lemma}

As for fixed $s$ and large $x$
\[
  \Gamma(s,x) \sim x^{s-1} e^{-x} \left( 1 + \frac{s-1}{x} + \frac{(s-1)(s-2)}{x^2}
  + \cdots \right)
\]
we see that the left hand side of~\eqref{eq:integral-test-estimate}
tends to zero for large $\alpha$ and fixed $j$ and $k$.

\begin{proof}
  The function $m \mapsto m^k e^{-2\alpha m}$ is monotonically
  decreasing for $m \geq k/(2\alpha)$. So we can apply the integral
  test
  \[
    \sum_{m = j}^\infty m^k e^{-2\alpha m} \leq j^k e^{-2\alpha j} +
    \int_{j}^\infty m^k e^{-2\alpha m} \, dm.
  \]
  Now using the definition of the incomplete gamma function after a
  change of variables yields the lemma.
\end{proof}

\subsection{Root systems}
\label{subsec:Root-Systems}

The shell $L(2)$ is called the root system of the even unimodular
lattice $L$, its elements are called roots. Witt classified in 1941
the possible root systems: These are orthogonal direct sums of the
irreducible root systems $A_n$ ($n \geq 1$), $D_n$ ($n \geq 4$),
$E_6$, $E_7$ and $E_8$.  The rank of a root system is the dimension of
the vector space it spans.  Let $e_1, \ldots, e_{n+1}$ be the
standard basis for $\R^{n+1}$. The root system $A_n$ is defined as
\[
\{\pm(e_i-e_j) : 1 \leq i < j \leq n+1\}.
\]
The root system $A_n$ has rank $n$, but lies in $\R^{n+1}$. It spans
the vector space $\R^{n+1} \cap \R(1,\ldots,1)^\perp \cong \R^n$.  In
the following we will consider $A_n $ as a subset in $\R^n $.  The
root system $D_n$ is defined as
\[
D_n= \{\pm(e_i\pm e_j) :  1 \leq i < j \leq n\}.
\]
Furthermore
\[
E_8 = D_8\cup \left\{\frac{e_1\pm\dots\pm e_8}{2}\right\},
\]
where we restrict the last set to all sums having an even number of
minus signs, and
\[
  E_7 = E_8\cap \R(e_7-e_8)^\perp \quad \text{and} \quad E_6 = E_7\cap
  \R(e_6-e_7)^\perp.
\]

All irreducible root systems form spherical $2$-designs, and we have
even spherical $4$-designs for $A_1$, $A_2$, $D_4$, $E_6$, $E_7$, and
$E_8$.

Let $R$ be a root system. Let $\sigma(x) = I_n - xx^{\sf T}$ be the
reflection at the hyperplane perpendicular to $x$. For all
$x, y \in R $ we have $\sigma(x) y \in R$, so that $R$ is invariant under
the reflection $\sigma(x)$. The group $W(R)$ generated by all reflections
$\sigma(x)$, with $x \in R$, is called Weyl group of the root system.

The Coxeter number $h$ of a root system $R$ with rank $n$ is defined
as $|R|/n$, the number of roots per dimension. For a root $r \in R$ we
denote by $n_0$ the number of roots $r'\in R$ with $r\cdot r' = 0$ and
by $n_1$ the number of roots $r' \in R$ with $r\cdot r'=1$. These
numbers $n_0$, $n_1$ do not depend on $r$ when $R$ is irreducible.

We summarize some properties of the irreducible root systems in
Table~\ref{table:Properties-Irreducible}.

\begin{table}[htb]
\small
\begin{tabular}{@{}ccccccc@{}}
\toprule
\textbf{name} & \textbf{rank} & \bm{$|R|$} & \bm{$n_0$} & \bm{$n_1$} & \bm{$h$} & \bm{$|W|$} \\
\midrule
$A_n$ & $n\geq 1$ & $n(n+1)$ & $(n-1)(n-2)$ & $2(n-1)$ & $n+1$ & $(n+1)!$ \\
\midrule
$D_n$ & $n\geq 4$ & $2n(n-1)$ & $2(n^2-5n+7)$ & $4(n-2)$ & $2(n-1)$ & $2^{n-1}n!$ \\
\midrule
$E_6$ & $6$ & $72$ & $30$ & $20$ & $12$ & $2^7 3^4 5$ \\
\midrule
$E_7$ & $7$ & $126$ & $60$ & $32$ & $18$ & $2^{10}3^4 5 7$ \\
\midrule
$E_8$ & $8$ & $240$ & $126$ & $56$ & $30$ & $2^{14}3^5 5^2 7$ \\
\bottomrule
\end{tabular}
\bigskip
\caption{Some properties of the irreducible root systems.}
\label{table:Properties-Irreducible}
\end{table}

\section{Strategy}
\label{sec:Strategy}

We compute the gradient and Hessian of
$L \mapsto \mathcal{E}(f_\alpha, L)$ at even unimodular lattices. For
this it is convenient to parametrize the manifold of rank $n$ lattices
having point density $1$ by positive definite quadratic forms of
determinant $1$.

The gradient and the Hessian of $\mathcal{E}(f_\alpha,L)$ at $L$ were
computed by Coulangeon and Sch\"urmann \cite[Lemma
3.2]{Coulangeon2012a}. Let $H$ be a symmetric matrix having trace zero
(lying in the tangent space of the identity matrix). We use the
notation $H[x] = x^{\sf T} H x$ and we equip the space of symmetric
matrices $\mathcal{S}^n$ with the inner product
$\langle A, B \rangle = \Tr(AB)$, where $A, B \in \mathcal{S}^n$. The
gradient is given by
\begin{equation}
\label{eq:gradient}
\langle  \nabla  \mathcal{E}(f_\alpha,L), H \rangle = -\alpha \sum_{x \in L \setminus \{0\}} H[x] e^{-\alpha \|x\|^2}.
\end{equation}

Now a sufficient condition for $L$ being a critical point is that all
shells of $L$ form spherical $2$-designs. Indeed, we group the sum
in~\eqref{eq:gradient} according to shells, giving
\[
\langle  \nabla  \mathcal{E}(f_\alpha,L), H \rangle =  -\alpha \sum_{r > 0} e^{-\alpha r^2} \sum_{x \in L(r^2)} H[x].
\]
Then for $r > 0$ every summand
\[
  \sum_{x \in L(r^2)} H[x]
  = \left\langle H, \sum_{x \in L(r^2)} x x^{\sf T}
   \right\rangle =  \frac{r^2 |X|}{n} \Tr(H) = 0
\]
vanishes because of~\eqref{eq:spherical-2-design} and because $H$ is
traceless. Hence, $L$ is critical.

This sufficient condition is fulfilled for all even unimodular
lattices in dimensions $8$, $16$, and $24$. This fact can be deduced
from the theory of theta functions with spherical coefficients and
modular forms as first observed by Venkov~\cite{Venkov1978a}. In
dimension $32$ this is no longer fulfilled in general but we can
identify cases where it is.

The Hessian is the quadratic form
\begin{equation}
\label{eq:Hessian}
\nabla^2 \mathcal{E}(f_\alpha,L)[H] = \alpha \sum_{x \in L\setminus\{0\}}
e^{-\alpha\|x\|^2} \left(\frac{\alpha}{2} H[x]^2 - \frac{1}{2}H^2[x]\right).
\end{equation}
Again grouping the sum according to shells we get
\begin{equation}
  \nabla^2 \mathcal{E}(f_\alpha,L)[H] = \alpha \sum_{r > 0} e^{-\alpha
    r^2} \sum_{x \in L(r^2)} \left(\frac{\alpha}{2} H[x]^2 - \frac{1}{2}H^2[x]\right).
\end{equation}
So it remains to determine the two sums
\begin{equation}
  \label{eq:two-sums}
  \sum_{x \in L(r^2)} H[x]^2 \quad \text{and} \quad \sum_{x \in
    L(r^2)} H^2[x].
\end{equation}
The second sum is easy to compute when $L(r^2)$ forms a spherical
$2$-design. In this case we have by~\eqref{eq:spherical-2-design}
\begin{equation}
\label{eq:spherical-2-design-H-squared}
  \sum_{x \in L(r^2)} H^2[x] = \left\langle H^2, \sum_{x \in L(r^2)}
      xx^{\sf T} \right\rangle  = \langle H^2,   \frac{r^2
      |L(r^2)|}{n} I_n \rangle =
  \frac{r^2 |L(r^2)|}{n} \Tr H^2.
\end{equation}
The first sum is only easy to compute when $L(r^2)$ forms even a
spherical $4$-design. Then (see \cite[Proposition
2.2]{Coulangeon2006a} for the computation)
\begin{equation}\label{eq:spherical-4-design-quadratic-form}
 \sum_{x \in L(r^2)} H[x]^2 = \frac{r^4 |L(r^2)|}{n(n+2)} 2 \Tr H^2.
\end{equation}
Together, when all shells form spherical $4$-designs, the
Hessian~\eqref{eq:Hessian} simplifies to
\begin{equation}
  \label{eq:Coulangeon-Hessian}
  \nabla^2 \mathcal{E}(f_\alpha,L)[H] = \frac{\Tr H^2}{n(n+2)}
  \sum_{r > 0} |L(r^2)| \alpha r^2 \left(\alpha r^2 - (n/2 + 1)\right) e^{-\alpha r^2}.
\end{equation}
Therefore, every $H$ with Frobenius norm $\langle H, H \rangle =
\Tr H^2 = 1$ is mapped to the same value, which implies that all the
eigenvalues of the Hessian coincide.

Sarnak and Str\"ombergsson \cite{Sarnak2006a}, see also Coulangeon
\cite{Coulangeon2006a}, showed that for $L = E_8, \Lambda_{24}$ the
Hessian $\nabla^2 \mathcal{E}(f_\alpha,L)[H]$ is positive for all
$\alpha >0$ which implies that $E_8, \Lambda_{24}$ are local minima
among lattices, for all completely monotonic potential
functions of squared Euclidean distance\footnote{This was one motivation for Cohn, Kumar, Miller,
  Radchenko, and Viazovska \cite{Cohn2019a} to prove their far
  stronger, global result.}.

The case when all shells form spherical $2$-designs but not spherical
$4$-designs requires substantially more work. This is our main
technical contribution. Then the Hessian has more than only one
eigenvalue. We determine these eigenvalues up to dimension $32$ by
considering the root system of $L$, that is the shell $L(2)$. Here the
quadratic form
\begin{equation}
\label{eq:crucial-quadratic-form}
  Q[H] = \sum_{x \in L(2)} H[x]^2
\end{equation}
will play a crucial role.

Indeed, consider again the first sum $\sum_{x \in L(r^2)} H[x]^2$
in~\eqref{eq:two-sums}. We decompose the polynomial $p_H(x) = H[x]^2$
into its harmonic components as in Lemma~\ref{lemma:polynomial-p-H}
and get
  \[
\sum_{x \in L(r^2)} p_H(x) = \sum_{x
      \in L(r^2)} p_{H,4}(x) + r^2 \sum_{x
      \in L(r^2)} p_{H,2}(x) + r^4 \sum_{x
      \in L(r^2)} p_{H,0}(x).
\]
Here the first sum equals
\[
\sum_{x \in L(r^2)} p_{H,4}(x) =  \sum_{x \in L(r^2)} H[x]^2 -
r^4 \frac{2}{(2+n)n} |L(r^2)| \Tr H^2,
\]
where we used Lemma~\ref{lemma:polynomial-p-H}
and~\eqref{eq:spherical-2-design-H-squared}. The second sum vanishes
because $L(r^2)$ is a spherical $2$-design and the third summand
equals
\[
  r^4 \sum_{x \in L(r^2)} p_{H,0}(x) = r^4  \frac{2}{(2+n)n} |L(r^2)| \Tr H^2.
\]

We can use theta series with spherical coefficients to determine the
first sum $\sum_{x \in L(r^2)} p_{H,4}(x)$
explicitly: $\Theta_{L,p_{H,4}}$ is a cusp form of weight $n/2+4$. In
dimension $16$, $24$, and $32$ there is (up to scalar multiplication)
only one cusp form of weight $n/2+4$. This is, respectively,
$\Delta$, $E_4 \Delta$ and $E_4^2 \Delta$. Their $q$-expansions
$ \sum_{m=0}^\infty b_m q^m $ all start by $0 + 1 \cdot q$.
Therefore, by
equating coefficients,
\[
\Theta_{L,p_{H,4}}(\tau)= \sum_{r>0}\sum_{x \in L(r^2)} p_{H,4}(x) q^{\frac{1}{2}r^2}  = c \sum_{m = 0}^\infty b_m q^m
\]
with
\[
c = \sum_{x \in L(2)} H[x]^2 - \frac{8}{(2+n)n} |L(2)| \Tr H^2.
\]
For $ r^2 = 2m $ it follows
\[
\sum_{x \in L(r^2)} H[x]^2  = c b_m   + 4m^2 \frac{2}{(2+n)n} |L(2m)| \Tr H^2.
\]
Hence, we only need to compute the eigenvalues
of~\eqref{eq:crucial-quadratic-form} to determine the signature of the
Hessian. When talking about eigenvalues of $Q$, we refer to the
eigenvalues of the Gram matrix with entries $b_Q(G_i,G_j)$, where
$ b_Q: \mathcal{S}^n \times \mathcal{S}^n \to \mathbb{R} $ is the
induced bilinear form
\begin{align}\label{eq:bil-form-b_Q}
	 b_Q(G,H) = \sum_{x\in L(2)}G[x] H[x]
\end{align}
and $(G_i)$ is an orthonormal basis of the space $\mathcal{S}^n$ with
respect to the inner product $\langle \cdot, \cdot \rangle$. If $H$ is
an eigenvector with eigenvalue $\lambda$, we have
\[
\sum_{x\in L(2)} H[x]^2 = \lambda \Tr H^2.
\]

Now let us put everything together.

\begin{theorem}
\label{thm:main}
  Let $L$ be an even unimodular lattice in dimension $n \leq 32$. Let
  \[
    \Theta_L(\tau) = \sum_{m = 0}^\infty a_m q^m \quad \text{with $a_m = |L(2m)|$}
  \]
  be the theta series of $L$ and let
  $
    \sum_{m=1}^\infty b_m q^m
$
be the cusp form of weight $n/2+4$ with $b_1 = 1$. Then all the eigenvalues of the Hessian
$\nabla^2 \mathcal{E}(f_\alpha, L)$ are given by
  \begin{equation}
    \label{eq:main}
    \begin{split}
& \frac{1}{n(n+2)}
    \sum_{m=1}^\infty  \left(b_m \frac{\alpha^2}{2}
  (\lambda n(n+2) - 8
  a_1)\right)  e^{-2 \alpha m}\\
& \; + \; \frac{1}{n(n+2)}
    \sum_{m=1}^\infty \left(a_m 2\alpha m \left(2\alpha m - (n/2 +
        1)\right)\right) e^{-2 \alpha m},
    \end{split}
  \end{equation}
where $\lambda$ is an eigenvalue of~\eqref{eq:crucial-quadratic-form}.
\end{theorem}

Note that this theorem also includes the case when all shells of $L$
form spherical $4$-designs like in~\eqref{eq:Coulangeon-Hessian}
because of~\eqref{eq:spherical-4-design-quadratic-form}. In this case
and when the parameter $\alpha$ is large enough,
then~\eqref{eq:Coulangeon-Hessian} is strictly positive, which shows
that $L$ is a local minimum for $f_{\alpha}$-potential energy.

\medskip

Similarly, because the growth of $a_m$ and $b_m$ is polynomial in $m$
and because of the estimate provided in
Lemma~\ref{lem:integral-test-estimate}, we see that the first summand,
$m = 1$,
\[
\frac{1}{n(n+2)}  \left(\frac{\alpha^2}{2} (\lambda n(n+2)) - 2 a_1 \alpha  (n/2 +
1) \right) e^{-2\alpha}
\]
dominates~\eqref{eq:main} for large $\alpha$. In particular, for
large~$\alpha$, the first summand is strictly positive if $\lambda$ is
strictly positive and the first summand is strictly negative if
$\lambda$ vanishes and if $a_1 \neq 0$. As the quadratic
form~\eqref{eq:crucial-quadratic-form} is a non-trivial sum of
squares, the eigenvalues cannot be strictly negative and some
eigenvalue is always strictly positive. From this consideration we get:

\begin{corollary}
  \label{cor:large-alpha}
  Let $L$ be an even unimodular lattice in dimension $n \leq 32$ which
  is critical for $f_\alpha$-potential energy.  For all large enough
  $\alpha$ the lattice $L$ is a local minimum if and only if all
  eigenvalues of~\eqref{eq:crucial-quadratic-form} are strictly
  positive. If one eigenvalue of~\eqref{eq:crucial-quadratic-form}
  vanishes and if $|L(2)| > 0$, then $L$ is a saddle point for all
  large enough $\alpha$.
\end{corollary}

\section{Eigenvalues of~\eqref{eq:crucial-quadratic-form}}
\label{sec:Eigenvalues}

In this section we shall compute the eigenvalues of the quadratic form
\eqref{eq:crucial-quadratic-form} $Q[H] = \sum_{x\in R}H[x]^2$, where
we write $R = L(2)$ for the root system of the lattice.

\subsection{Irreducible root systems}

First we consider the case when $R$ is an irreducible root system of
type $A$, $D$, or $E$.

\begin{theorem} \label{Theorem:Eigenvalues} Let $R$ be an irreducible
  root system of type $A$, $D$, or $E$.  The quadratic form
  $Q[H] = \sum_{x\in R}H[x]^2$ has the following eigenvalues:
  \begin{center}
    \small
		\begin{tabular}{@{}ccc@{}}
			\toprule
			\textup{\textbf{root system}} & \textup{\textbf{eigenvalue}} & \textup{\textbf{multiplicity}} \\
			\midrule
			\multirow{3}{*}{$A_n,~n\geq 1$}   &$4h = 4(n+1)$ & $1$\\
			\addlinespace
			& $2(n+1)$  & $n$, for $n\geq 2$\\
			\addlinespace
			& $4$       & $n(n-1)/2 -1$, for $n\geq 2$\\
			\midrule
			\multirow{3}{*}{$D_n,~n\geq 4$}    & $4h = 8(n-1)$& $1$\\
			\addlinespace
			& $4(n-2)$     & $n-1$\\
			\addlinespace
			& $8$          & $n(n-1)/2$\\
			\midrule
			\multirow{2}{*}{$E_6$}    & $4h = 48$   &$1$\\
			\addlinespace
			& $12$       & $20$\\
			\midrule
			\multirow{2}{*}{$E_7$}    &$4h = 72 $    &$1$\\
			\addlinespace
			& $16$       & $27$\\
			\midrule
			\multirow{2}{*}{$E_8$}    &$4h = 120$    &$1$\\
			\addlinespace
			& $24$      & $35$\\
			\bottomrule
		\end{tabular}
                \end{center}
              \end{theorem}

We will embed the proof of Theorem~\ref{Theorem:Eigenvalues} in the
framework of representation theory.\footnote{In the following we apply
  concepts of unitary representations over the complex numbers, but
  note that all representations involved can in fact be defined over
  the reals.} The Weyl group $W$ of the root system $R$ acts on the space of
symmetric matrices $\mathcal{S}^n$ by conjugation
\begin{align*}
	W \times \mathcal{S}^n &\to \mathcal{S}^n   \\
	(S,H)  &\mapsto  S H S^{\sf T}.
\end{align*}
This turns $ (\mathcal{S}^n, \langle \cdot, \cdot \rangle )$ into a
unitary representation of $ W $, meaning that the action of $ W $
preserves the inner product
$ \langle \cdot , \cdot \rangle $.

Then the
bilinear form $ b_Q $, defined in \eqref{eq:bil-form-b_Q}, is invariant under the action of the Weyl group
$ W $, that is $ b_Q(SGS^{\sf T},SHS^{\sf T}) = b_Q(G,H) $ for all
$ S \in W $.
Due to the Riesz representation theorem, there is a linear map
$T : \mathcal{S}^n \to \mathcal{S}^n$ such that
\[
	b_Q(G,H) = \langle G, T(H) \rangle
\]
and the eigenvalues of the Gram matrix of $b_Q$ coincide with the
eigenvalues of $T$.
Since $ b_Q $ is invariant under the action of $ W $, the map $ T $ commutes with the action of
$ W$, i.e.
\begin{equation}
\label{eq:commutativity-property-of-E}
T(S H S^{\sf T}) = ST(H)S^{\sf T} \quad   \textup{ for all }    S \in W,
\end{equation}
hence, $ T $ is intertwining the representation
$ (\mathcal{S}^n, \langle \cdot, \cdot \rangle ) $ of the Weyl group
$ W $ with itself.

Instead of only considering the specific map $ T $ above, we determine
the common eigenspaces of all intertwiners that intertwine the
representation on $\mathcal{S}^n$ with itself. As these eigenspaces will turn out to be inequivalent,
Schur's lemma implies that these eigenspaces are exactly the pairwise
orthogonal, irreducible, $W$-invariant subspaces of $ \mathcal{S}^n $.

\subsection{Peter-Weyl theorem for irreducible root systems}

This gives rise to Theorem \ref{theorem:Peter-Weyl-For-Sn}, which is a
\emph{Peter-Weyl theorem} for the representation
$ (\mathcal{S}^n, \langle \cdot, \cdot \rangle ) $ of the Weyl group
$ W $ of an irreducible root system.

To state the theorem, we need to fix some notation, based on the
definition of root systems in Section \ref{subsec:Root-Systems}.  We
consider $ A_n $ as a root system in $ \R^n $ and, by slight abuse of
notation, we write $e_i -e_j$ for the corresponding root in $ \R^n$.
Moreover, define the symmetric bilinear operator
$M : \R^n\times \R^n \rightarrow \mathcal{S}^n$ by
\[
M(x,y) = xy^{\sf T} + yx^{\sf T}.
\]
The action of the Weyl group on $ M $ is given by
\[
SM(x,y)S^{\sf T} = M(Sx,Sy), \qquad S \in W.
\]
Furthermore, set
\[
M(x) = \frac{1}{2}M(x,x) = xx^{\sf T}
\]
and
\[
P_i = \sum_{ j \in \{1, \ldots, n+1\} \setminus \{i\} }M(e_i-e_j) - 2I_n.
\]

\begin{theorem} [Peter-Weyl for irreducible root systems]
\label{theorem:Peter-Weyl-For-Sn}
The space of symmetric matrices can be decomposed into the following
$W$-invariant, irreducible, inequivalent subspaces:
\begin{enumerate}[label=(\roman*)]
\item \label{peter-weyl-i} For $ R = A_n,~n\geq 2 $
\[
    \mathcal{S}^n = \lspan  \{I_n\} \perp U_1(A_n) \perp U_2(A_n),
\]
  where
  \begin{align*}U_1(A_n) &= \lspan  \{ M(x,y) :  x,y \in A_n, \, x\cdot y = 0  \},   \\
    U_2(A_n) &= \lspan \{ P_i : i = 1, \ldots, n+1 \}.
  \end{align*}
\item \label{peter-weyl-ii} For $ R = D_n,~n\geq 5$
\[
    \mathcal{S}^n = \lspan  \{I_n\} \perp U_1(D_n) \perp U_2(D_n),
\]
  where
  \begin{equation}
	U_1(D_n) =  \{M\in \mathcal{S}^n : M_{ii}=0, \, 1\leq i\leq
    n\} \label{eq:Off-Diagonal-EIgenspace}.
\end{equation}
and
\begin{equation}
	U_2(D_n) = \{\textup{diag}(d_1,\ldots,d_n) : d_1,\ldots,d_n\in\R,\,d_1+
	\cdots + d_n =0\} \label{eq:Diagonal-EIgenspace}.
\end{equation}
For $ n = 4 $ the space $ U_1(D_4) $ further splits into two
irreducible subspaces
\begin{equation}
  \label{eq:special-case-D-4}
  \begin{split}
  U_1(D_4) \; = \;   & \left \{ \begin{pmatrix}
			0 & a & b & c \\
			a & 0 & c & b \\
			b & c & 0 & a \\
			c & b & a & 0
                      \end{pmatrix}   \, : \,  a,b,c \in \R \right  \}
                    \\
 & \perp \quad \left  \{  \begin{pmatrix}
			0 & a & b & -c \\
			a & 0 & c & -b \\
			b & c & 0 & -a \\
			-c & -b & -a & 0
		\end{pmatrix}  \, : \,  a,b,c \in \R \right  \}    .
	\end{split}
\end{equation}
\item \label{peter-weyl-iii} For $ R \in \{E_6,E_7,E_8\} $
\[
    \mathcal{S}^n = \lspan  \{ I_n  \} \perp \mathcal{T}^n_0,
\]
  where $ \mathcal{T}^n_0$ is the space of traceless symmetric
  $ n \times n $ matrices.
\end{enumerate}
\end{theorem}

\begin{remark}
  The proofs of \ref{peter-weyl-i} and \ref{peter-weyl-ii} will be
  based on the representation theory of the symmetric
  group\footnote{The authors would like to thank one of the anonymous
    referees for the suggestion and a detailed sketch of this
    approach.}(see \cite[Chapter 4]{Fulton1991} for details). In fact,
  the decompositions are immediate consequences of the representation
  theory of the symmetric group, most of the work lies in the explicit
  description of the irreducible subrepresentations, as we need these,
  for the explicit calculation of the eigenvalues in
  Theorem~\ref{Theorem:Eigenvalues}.

  We will give an elementary proof of \ref{peter-weyl-iii} in Section
  \ref{subsec:E-n}. However, as one of the anonymous referees pointed
  out, this could also be done by computing the explicit characters of
  the representation, as it was already done in the literature.
  See \cite{Frame1951} for the case $E_6$ and $E_7$, and
  \cite{Frame1970}, for the case $E_8$.
\end{remark}

The main ingredient is a decomposition formula for a representation of
$\mathfrak{S}_{n+1}$, the symmetric group on $n+1$ symbols. We write
\[
  U = \lspan\{e\},
\]
where $e$ is the all ones vector, for the trivial representation and
\begin{equation}
	\label{eq:V_n+1}
	V_{n+1} = \left\{ v \in \R^{n+1}  \, : \,  \sum_{i = 1}^{n+1} v_i = 0 \right\} = U^\perp
\end{equation}
for the standard representation of $\mathfrak{S}_{n+1}$. Clearly $U$
and $V_{n+1}$ are orthogonal as representations. Furthermore, both are
irreducible: they are the cases of a standard principle to construct
the irreducible representations of $\mathfrak{S}_{n+1}$ via Young
symmetrizers, which give a one-to-one correspondence between
partitions of $n+1$ and irreducible representations of
$\mathfrak{S}_{n+1}$ \cite[Theorem 4.3]{Fulton1991}.

One then obtains the decomposition\footnote{C.f. Exercise
  \cite[4.19]{Fulton1991}, which can be solved by showing that the
  representation $ \Sym^2(V_{n+1}) $ is equivalent to the
  representation $ U_{(n-2,2)} $, defined on
  \cite[P.~54]{Fulton1991}. This can be done by explicitly computing
  the character of $ \Sym^2(V_{n+1}) $ (see
  \cite[Chapter 2]{Fulton1991}) and $ U_{(n-2,2)} $ (see
  \cite[Eq.~4.33]{Fulton1991}). A decomposition of $ U_{(n-2,2)} $
  into irreps is given in the last displayed equation of
  \cite[P.~57]{Fulton1991}.}
\begin{equation}
  \label{eq:sym_decomposition}
  \Sym^2(V_{n+1}) \cong U \oplus  V_{n+1} \oplus V_{((n+1)-2,2)},
\end{equation}
where $V_{((n+1)-2,2)}$ is another irreducible representation\footnote{This is the irreducible representation corresponding to the partition $((n+1)-2,2)$ of $n+1$ of $\mathfrak{S}_{n+1}$, a Specht module.} of $\mathfrak{S}_{n+1}$.

\subsection{$ A_n $}\label{subsec:A_n}

We will begin with \ref{peter-weyl-i}. It is well known that $W(A_{n+1}) \cong \mathfrak{S}_{n+1}$ and we can explicitly describe the action of $W(A_n)$ in terms of the action of $\mathfrak{S}_{n+1}$ by permutation matrices via the identification $\mathcal{S}^n \cong \textup{Sym}^2(V_{n+1})$. For this, we explicitly write
\[
\textup{Sym}^2(V_{n+1}) = \{  A \in \mathcal{S}^{n+1} \, : \, Ae = 0 \},
\]
where, again, $ e $ is the all-ones vector. This can be done by identifying the root projectors $  xx^\top$ with $ x \in A_n $ with the projectors $ M(e_i-e_j)  $ with $ e_i-e_j \in \R^{n+1} $.
Let $ \mathfrak{S}_{n+1} $ be the symmetric group on $ n+1 $ symbols. Define a group action of $ \mathfrak{S}_{n+1} $ on $ \R^{n+1} $ via
\begin{align}\label{eq:ActionSymmetricGroup}
	\mathfrak{S}_n \times \R^{n+1} \to \R^{n+1}, \quad \sigma(v) :=  (v_{\sigma(1)}, \ldots, v_{\sigma(n+1)}).
\end{align}
For a Weyl group generator $ S = I_{n+1} - aa^{\sf \top} $ with $ a = e_i - e_j $ one can straightforwardly verify that $ S $ is a permutation matrix that swaps the entries $ v_i $ and $ v_j $:
\[
Sv = \sigma(v), \quad \sigma = (i \, \,  j).
\]
As $ \mathfrak{S}_n $ is generated by $ 2 $-cycles, it follows that $ W(A_{n}) $ is a matrix representation (by permutation matrices) of $ \mathfrak{S}_{n+1} $ acting on $\textup{Sym}^2(V_{n+1})  $.

This identification enables us to use decomposition \eqref{eq:sym_decomposition}
and at this point we, in principle, have already found the decomposition proposed in the theorem. Clearly $U \cong \lspan \{I_n\}$. Furthermore, below we will show that $U_1(A_n), U_2(A_n)$, as given in the theorem, are indeed subrepresentations of $ W(A_n) \cong \mathfrak{S}_{n+1} $ orthogonal to each other and $\lspan\{I_n\} \cong U$. We now proceed by comparing dimensions of the remaining summands. By the hook length formula \cite[4.12]{Fulton1991} we find $\dim(V_{n+1}) = n$ and $\dim(V_{((n+1)-2,2)}) = (n+1)(n-2)/2$. In Lemma \ref{lemma:dimension-U_2(A_n)} we will show that $\dim(U_2(A_n)) = n = \dim(V_{n+1})$, it then
follows that $ U_2(A_n) \cong V_{n+1} $. This also implies that $U_1(A_n) \cong V_{(n-2,2)}$, as the orthogonality of $U_1(A_n)$ and $U_2(A_n)$ implies that $U_1(A_n)$ is a subrepresentation of $V_{(n-2,2)}$, which, by the irreducibility of the latter, implies equivalence.

Therefore the following list of equivalences of representations is valid
\[
  U \cong \lspan \{I_n\},\ V_{n+1} \cong U_2(A_n),\ V_{(n-2,2)} \cong U_1(A_n),
\]
which then, since $U$, $V_{n+1}$, and $V_{((n+1)-2,2)}$ are irreducible, finishes the proof of part \ref{peter-weyl-i} of the theorem.

We will conclude this part of the proof by showing that $U_1(A_n), U_2(A_n)$ are indeed subrepresentations of $W(A_n)$, are orthogonal to each other and computing $\dim(U_2(A_n)) = n$ as used above.

We first show orthogonality. It is straightforward to check that all
operators in $ U_i(R) $ for $R \in \{D_n,A_n\} $ and $ i = 1,2 $ are
traceless, so $ \lspan \{I_n\} \perp U_i(R) $.

For $ U_1(A_n) \perp U_2(A_n)$, we need to check that for orthogonal
roots $ x,y \in A_n $
\begin{equation}
	\label{eq:OrthogonalityForAn}
	0 =\langle P_i, M(x,y)\rangle = 2\sum_{ j \in \{1,\ldots,n+1\} \setminus \{i\} }(x\cdot (e_i -e_j))(y\cdot (e_i -e_j)).
\end{equation}
Every summand of the right hand side of \eqref{eq:OrthogonalityForAn} is zero, if $ x = e_k-e_l $ and $ y = e_s -e_t $ for $ k,l,s,t \neq i $.
Otherwise, if $ x = \pm (e_i-e_k) $ and $ y = e_s - e_t $, then $(x\cdot (e_i -e_j))(y\cdot (e_i -e_j))$ is only non-zero, if $ j = s $ or $ j = t $.

Then we get
\[
(\pm(e_i-e_k)\cdot (e_i-e_s))((e_s-e_t)\cdot (e_i-e_s)) = \mp 1
\]
and
\[
(\pm(e_i-e_k)\cdot (e_i-e_t))((e_s-e_t)\cdot (e_i-e_t)) = \pm 1.
\]
Thus, the sum of the right hand side of \eqref{eq:OrthogonalityForAn}
is zero, which implies that the inner product is zero. Hence, all
spaces in \ref{peter-weyl-i} are orthogonal.

Next, we show that the spaces are invariant under the action of the Weyl group.
If $ x,y $ are orthogonal roots, then for $S\in W$ the roots
$ Sx, Sy $ are orthogonal as well, because the Weyl group preserves
orthogonality.  This directly implies the invariance of $ U_1(A_n)
$.
For the invariance of $ U_2(A_n) $ it suffices to observe that
\begin{align*}
	S_{e_i-e_j} P_k (S_{e_i-e_j})^{\mathsf{\top}} = \begin{cases}
		P_j, & \text{ if } i = k  \\
		P_i, & \text{ if } j = k \\
		P_k, & \text{ otherwise.}
	\end{cases}
\end{align*}

As a last step we compute the dimension of the space $ U_2(A_n) $.
\begin{lemma}\label{lemma:dimension-U_2(A_n)}
		 For $ n \ge 2 $ it holds that
		$
			\dim U_2(A_n) = n.
		$
\end{lemma}

\begin{proof}
	By summing the generators $ P_i $ of $ U_2(A_n) $ we obtain
	\[
	\sum_{i=1}^{n+1} P_i =  \sum_{x\in A_n}xx^{{\sf T}} -2(n+1)I_n,
	\]
	because each root projector $ xx^{\sf T}$, with $x \in A_n $,
	occurs in exactly two operators $ P_i $ and the roots $ x $
	and $ -x $ correspond to the same projector
	$ xx^{\sf T} = (-x)(-x)^{\sf T} $.  Since irreducible root systems
	are spherical $2$-designs, \eqref{eq:spherical-2-design}
	implies that
	\[
	\sum_{x\in R}xx^{{\sf T}} = 2hI_n = 2(n+1)I_n.
	\]
	Hence, $ \sum_{i=1}^{n+1} P_i =0 $, and so the matrices $P_i$
	are linearly dependent. We now show that the matrices
	$P_1,\ldots,P_n$ are linearly independent, implying
	$ \dim U_2(A_n) = n $. Suppose we have
	$\lambda_1,\ldots,\lambda_n\in \R$ with
	\[
	\sum_{i=1}^n \lambda_i P_i= 0.
	\]
	Let $\lambda = \lambda_1+\cdots+ \lambda_n$. We can write this equation as
	\[
	\sum_{i=1}^n \sum_{j\in \{1,\ldots,n+1\}\setminus\{i\}} \lambda_i M(e_i -e_j) + 2\lambda I_n = 0.
	\]
	For $ i \neq j $, the projector $M(e_i -e_j)$ appears as a summand in $P_i$ and $M(e_j-e_i)$ in $P_j$. Because $M(e_i -e_j)=M(e_j-e_i)$, rearranging the terms yields
	\[
	\sum_{1\le i < j\le n} (\lambda_i+\lambda_j) M(e_i -e_j) + \sum_{j=1}^n \lambda_j M(e_j-e_{n+1}) + 2\lambda I_n = 0.\]
	As by \eqref{eq:spherical-2-design},
	\begin{align*}
		I_n = \frac{1}{2(n+1)}\sum_{x\in R} xx^{\sf T}&= \frac{1}{2(n+1)}\sum_{\substack{i,j = 1,\ldots,n+1\\i\neq j}}M(e_i -e_j)    \\
		&=    \frac{1}{n+1}\sum_{1\leq i <j \leq n+1}M(e_i-e_j),
	\end{align*}
	this becomes

	\begin{align}\label{eq:IndependenceOfPi}
		&\sum_{\substack{i,j=1,\ldots,n\\i\neq j}} \left (\lambda_i+\lambda_j+\frac{2\lambda}{n+1} \Big) M(e_i -e_j)
		+ \sum_{j=1}^n \Big(\lambda_j+\frac{2\lambda}{n+1} \right) M(e_j-e_{n+1})
		=0.
	\end{align}
	Because the root projectors $ \{M(e_i-e_j) \, : \, 1 \le i < j
        \le n+1\} $ are linearly independent\footnote{This also
          follows from the fact that the root lattice is perfect and
          the number of root projectors coincides with the dimension of $ \mathcal{S}^n $.},  \eqref{eq:IndependenceOfPi} implies
	that
	\begin{align*}
		\lambda_i + \lambda_j + \frac{2\lambda}{n+1} &= 0,  \quad 1\le i \neq j \le n,  \\
		\text{and} \qquad \qquad \lambda_j+\frac{2\lambda}{n+1}&= 0, \quad  j=1,\ldots,n.
	\end{align*}
	By subtracting the equations, it follows that
	$\lambda_1=\ldots=\lambda_n =0$.
\end{proof}

\subsection{$ D_n $}

We will proceed with \ref{peter-weyl-ii}. The overall strategy is the same as in the $A_n$ case. On the abstract level we consider the representation
\[
  \mathcal{S}^n \cong \textup{Sym}^2(\R^n) = \textup{Sym}^2(U + V_n).
\]
We first obtain a decomposition of $\textup{Sym}^2(U + V_n)$ with respect to the action of the subgroup $\mathfrak{S}_n < W(D_n)$.

To this end, we first note that
\[
  \textup{Sym}^2(U + V_n) \cong \bigoplus_{a,b:\ a+b = 2} \textup{Sym}^a(U) \otimes \textup{Sym}^b(V_n) \cong U \oplus V_n \oplus \textup{Sym}^2(V_n)
\]
and the latter decomposes by \eqref{eq:sym_decomposition}, thus
\[
  \textup{Sym}^2(U + V_n) \cong U \oplus U \oplus V_n \oplus V_n \oplus V_{(n-2,2)}
\]
as $\mathfrak{S}_n$-representations.

Now we examine how these (irreducible) $\mathfrak{S}_n$-subrepresentations behave under the action of $W(D_n)$, by directly comparing them to the modules given in the theorem.

First we show that the spaces in \ref{peter-weyl-ii} are indeed representations of $ W(D_n) $.
It is obvious that the spaces in \ref{peter-weyl-ii} are orthogonal. To verify that the spaces are indeed subrepresentations, note that for $ S_\alpha  $ for $ \alpha^- = e_i -e_j$, $ \alpha^+ = e_i+e_j $  and $ \sigma = (i \, \, j) \in \Sigma_n $ we have
\begin{align*}
	S_{\alpha^-}M(e_k,e_\ell) S_{\alpha^-}^{\mathsf{\top}} &= M(e_{\sigma(k)}, e_{\sigma(\ell)}), \\ S_{\alpha^+}M(e_k,e_\ell) S_{\alpha^+}^{\mathsf{\top}} &=  M((-1)^{\delta_{k \in \{i,j\}}}e_{\sigma(k)}, (-1)^{\delta_{ \ell \in \{i,j\}}} e_{\sigma(\ell)}),
\end{align*}
implying that $ W(D_n) $ preserves $ I_n $ and maps the off-diagonal, respectively diagonal entries of a matrix to its off-diagonal, respectively diagonal entries.
Hence, the spaces $ U_1(D_n) $, $ U_2(D_n) $ and $ \lspan \{I_n\} $ are invariant under $ W(D_n) $.
The special case $ D_4 $ where $ U_1(D_4) $ decomposes further into two $ 3 $-dimensional invariant subspaces will be treated at the end of this section. 
Now as $\mathfrak{S}_n$-representations we get (i.e. by comparing dimensions)
\[
  \lspan\{I_n\} \cong U,\ U_2(D_n) \cong V_n,
\]
and, since they are already irreducible with respect to $\mathfrak{S}_n$, that these are irreducible $W(D_n)$-subrepresentations.
Furthermore, by orthogonality, this implies
\[
  U_1(D_n) \cong U \oplus V_n \oplus V_{(n-2,2)}.
\]
We are left to show that $U_1(D_n)$ is irreducible for $n \geq5$ and to obtain a decomposition into irreducible subrepresentations for $n =4$.

It is easy to see that, with respect to the action of $ \mathfrak{S}_n$,
\begin{align*}
	U \oplus V_n  \cong L &\coloneqq \lspan  \left \{   M \Big (e_i, \sum_{j \in \{1,\ldots,n\} \setminus{\{i\}}}e_j \Big ) \, : \, i  = 1,\ldots, n \right \}   \\
	&= \left  \{ \sum_{1 \le i < j \le n} (a_i+a_j) M(e_i,e_j) \, : \; a_1,\ldots, a_n \in \R  \right  \}
	 \subset U_2(D_n)
\end{align*}
and
\begin{align*}
	U \cong L_1 \coloneqq\lspan \left \{\sum_{1 \le i < j \le n} M(e_i,e_j)   \right \}.
\end{align*}
Hence, by orthogonality of $ U $ and $ V_n $,
\begin{align*}
	V_n \cong L_1^\perp \coloneqq \left  \{ \sum_{1 \le i < j \le n} (a_i+a_j) M(e_i,e_j) \, : \; a_1,\ldots, a_n \in \R, \, \sum_{i = 1}^n a_i = 0  \right  \}.
\end{align*}
If $ U_2(D_n) $ is not an irreducible $ W(D_n) $-representation, then, by Maschke's theorem, either $ L_1, L_1^\perp $ or $ L = L_1 \perp L_1^\perp $ is an irreducible $ W(D_n) $-representation.

We can directly see that $L_1$ and $L$ are not even $W(D_n)$-invariant: considering the action of the element $\alpha = e_1+e_2 \in W(D_n)$ gives
\[
	 S_\alpha \Big  (\underbrace{\sum_{1 \le i < j \le n} M(e_i,e_j) }_{\in L_1 \subset L} \Big ) S_\alpha^{\mathsf \top}  = M(e_1,e_2) - \sum_{i \in \{1,2\}} \sum_{k = 3}^n M(e_i,e_k) + \sum_{3 \le i,j \le n} M(e_i,e_j) \notin L,
\]
by showing that a certain system of linear equations has no solutions.

We are left with the case of $ L_1^\perp $ to consider. Here we fix the element
\[
	X \coloneqq  M(e_1,\sum_{j \in \{2,\ldots,n\}}e_j) -  M(e_2,\sum_{j \in \{1,\ldots,n\} \setminus \{2\}}e_j) \in L_1^\perp.
\]
Now, choosing $ \alpha = e_3+e_4 $, we can show that $ S_\alpha X S_\alpha \notin L_1 \oplus L_1^\perp $ for $ n \ge 5 $, again by considering a system of linear equations.

However, if $n = 4$, the system allows for a solution and the space $ L_1^\perp $ can be written as
\begin{align*}
	L_1^\perp = \left  \{  \begin{pmatrix}
		0 & a & b & -c \\
		a & 0 & c & -b \\
		b & c & 0 & -a \\
		-c & -b & -a & 0
	\end{pmatrix}  \, : \,  a,b,c \in \R \right  \},
\end{align*}
which can be shown to be invariant under $ W(D_4) $. Thus, $  L_1^\perp $ is irreducible and $ U_2(D_4) $ splits into two $ W(D_4) $-irreducible subspaces as
\begin{align*}
	U_2(D_4) = L_1^\perp \oplus L_2, \quad L_2= \left \{ \begin{pmatrix}
		0 & a & b & c \\
		a & 0 & c & b \\
		b & c & 0 & a \\
		c & b & a & 0
	\end{pmatrix}   \, : \,  a,b,c \in \R \right  \}
\end{align*}
with $ L_2 \cong  U \oplus V_{(n-2,2)}$.

It remains to prove that the irreducible subspaces for the  special case $ D_4$ are inequivalent, despite having the same dimension.

We will do this by showing a more general statement, that is, if  $ T $ is an intertwiner with respect to the action of $ W(D_n) $, then $ M(x,y) $ is an eigenvector of $ T $ for orthogonal roots $ x,y \in D_n $.

In the case of $ D_4 $, all three subspaces $ U_2(D_n) $, $ L_1^\perp $ and $ L_2 $ contain an operator $ M(x,y) $ for orthogonal roots $ x,y \in D_4 $.
This shows in particular that the intertwiner $ T $ is either identically zero on one of the three subspaces or $ U_2(D_n) $, $ L_1^\perp $ and $ L_2 $ or $ T $ must preserve the three subspaces. By Schur's lemma, this implies that they are inequivalent.

To see this, note that for $ \sigma(x)\in W $ it holds that
\begin{align*}
	\sigma(x) M(x,y)  \sigma(x)^{\sf T} = \sigma(y)M(x,y) \sigma(y)^{\sf T}= - M(x,y),
\end{align*}
so $M(x,y)$ is contained in the subspace
\begin{align*}
	U_{xy} := \{ X \in S^n \, : \,
	\sigma(x) X \sigma(x)^{\sf T} = \sigma(y)^{\sf T} X \sigma(y)^{\sf T} = -X	\}.
\end{align*}
Let $ X \in U_{xy} $. Since $ T $ commutes with the action of $ W$, it follows
\begin{align*}
	\sigma(x) T(X)  \sigma(x)^{\sf T} = T(\sigma(x) X \sigma(x)^{\sf T}) = - T(X) = \sigma(y) X \sigma(y)^{\sf T},
\end{align*}
hence $ T(U_{xy})  \subseteq U_{xy} $.
Now, consider the $ M(x,y) $ with $ x = e_1+e_3 $ and $ y = e_3+e_4 $ and assume that
$  X = \sum_{1 \le i \le j \le n} c_{ij} M(e_i,e_j)  \in U_{xy}$.
Due to
\[
\sigma(x) e_i = \begin{cases}
	-e_i, & \text{if } i = 1,2  \\
	e_i, & \text{otherwise}
\end{cases}
\qquad
\sigma(y) e_i = \begin{cases}
	-e_i, & \text{if } i = 3,4  \\
	e_i, & \text{otherwise}
\end{cases}
\]
it follows that
\begin{align*}
	-X = \sigma(x) X \sigma(x)^{\mathsf{\top}} =& M(e_1,e_2) + M(e_1,e_1) + M(e_2,e_2)   \\
	&-\sum_{i > 2} c_{2i} M(e_1,e_i) + c_{1i} M(e_2,e_i)    + \sum_{i,j > 2} c_{ij} M(e_i,e_j),
\end{align*}
hence $ c_{1i} = c_{2i} $ and $ c_{ij} = 0 $ for all other cases.
Acting with $ \sigma(y) $ on $ X $ yields
\begin{align*}
	- X &=  \sigma(y) \Big (\sum_{i > 2} c_{1i}    (M(e_1,e_i ) + M(e_2,e_i)  )    \Big) \sigma(y)^{\mathsf{\top}}    \\
	&= -c_{14}M(e_1,e_3) - c_{13}  M(e_1,e_4) - c_{14} M(e_2,e_3) - c_{13} M(e_2,e_4) + \sum_{i > 5} c_{1i} M(e_1,e_i),
\end{align*}
so $ c_{14} = c_{13} $ and $ c_{1i} = 0 $ for $ i \neq 3,4 $. Hence, $ X = cM(x,y) $ for some constant $ c \in \R $ and $ U_{xy} $ is one-dimensional. As $ T(U_{xy}) \subseteq U_{xy} $, this shows that $ M(x,y) $ is an eigenvector of the intertwiner $ T $.
The argument for general orthogonal roots $ x,y \in D_n$ follows in the same manner.

\subsection{$ E_n $} \label{subsec:E-n}

To give an elementary proof that $ \mathcal{T}_0^n $ is irreducible with respect to the action of $ W(E_n) $, we will
use \ref{peter-weyl-ii} of Theorem \ref{theorem:Peter-Weyl-For-Sn}.

 In all three cases we consider the embedding of the root system $ E_n $
into $ \R^8 $, as defined in Section~\ref{subsec:Root-Systems}.  For
$ n \in \{6,7\} $ the space $ \mathcal{T}^n_0 $ embeds into
$ \lspan \{xx^{\sf T} \, : \, x \in E_n \} \subset \mathcal{S}^8$ via
\begin{align*}
	\mathcal{T}^n_0  \cong \begin{cases}
		\{ X \in \mathcal{T}^8_0 \, : \, X(e_7-e_8) = 0, \, X(e_6-e_7) = 0 \}, &\text{if } n = 6  \\
		\{ X \in \mathcal{T}^8_0 \, : \, X(e_7-e_8) = 0 \}, &\text{if } n = 7.
	\end{cases}
\end{align*}
Further, we embed the root systems $ D_n $ for $ n \le 8 $ into $ \R^8 $ by adding zero coordinates to the roots.
Let $ D_{s_n} $ be the largest root system of type $ D $ that is contained in $ E_n $, that is $ D_{s_6} = D_5, \, D_{s_7} = D_6 $ and $ D_{s_8} = D_8 $.
Since $ W(D_{s_n}) $ is a subgroup of the Weyl group $ W(E_n) $, Schur's lemma implies that every intertwiner $ T $ with respect to $ W(E_n) $ is a scalar multiple of the identity on $ U_i(D_{s_n}) $.The intertwiner commutes with the group action, thus, it is also a scalar multiple on
\[
W(E_n) \cdot U_i(D_{s_n}) := \{  SXS^{\mathsf{\top}} \, : \, X \in U_i(D_{s_n}) \},
\]
so $W(E_n) \cdot U_i(D_{s_n})   $ is an irreducible subspace for the action of $ W(E_n) $. Hence, to prove the irreducibility of $ \mathcal{T}_0^n $, it suffices to prove that
\begin{equation}\label{eq:SufficientConditionIrredEn}
	\mathcal{T}_0^n = W(E_n) \cdot U_2(D_{s_s}).
\end{equation}
First, we show that the two orbits $ W(E_n) \cdot U_i(D_{s_s}) $ collapse to one subspaces under the action of $ W(E_n) $:
\begin{lemma}
	\label{lemma:TwoOrbitsBecomeOne}
	It holds that $ U_1(D_{s_n}) \subset W(E_n) \cdot U_2(D_{s_n}) $ and
	in particular,
	\begin{align*}
		T^{s_n}_0 \cong \lspan  \{ U_1(D_{s_n}),U_2(D_{s_n}) \} \subset  W(E_n) \cdot U_2(D_{s_n}),
	\end{align*}
	where the first equivalence is a consequence of Theorem~\ref{theorem:Peter-Weyl-For-Sn} \ref{peter-weyl-ii}.
\end{lemma}

The lemma already shows the identity \eqref{eq:SufficientConditionIrredEn} for
$ n = 8 $ and therefore the irreducibility of $ \mathcal{T}^8_0 $ with
respect to the action of $ W(E_8) $.

\begin{proof}
	It suffices to show that for $ M(e_1+e_2,e_3-e_4) \in U_1(D_{s_n})$ and $  M(e_1+e_2,e_1-e_2) \in U_2(D_{s_n})$ it holds that
	\begin{align*}
		M(e_1+e_2,e_3-e_4)  \in W(E_n)\cdot M(e_1-e_2,
		e_1+e_2).
	\end{align*}
	We have
	\begin{align*}
		e_1-e_2 = (1,-1,0,0,0,0,0,0) \quad \overset{\sigma(x_1)}{\longmapsto} \quad  &\frac{1}{2}  (1, -1, -1, 1, 1, 1, 1, 1) =:   y  \\
		\textup{for}  \quad
		x_1 = &\frac{1}{2}  (1, -1, 1, -1, -1, -1, -1, -1)   \in E_n.
	\end{align*}
	Moreover,
	\begin{align*}
		y =\frac{1}{2}  (1, -1, -1, 1, 1, 1, 1, 1)   \quad \overset{\sigma(x_2)}{\longmapsto}   \quad   &(0,0,-1,1,0,0,0,0) = -(e_3-e_4)      \\
		\textup{for}  \quad x_2 =&\frac{1}{2}(1,-1,1,-1,1,     1,1,1)   \in E_n.
	\end{align*}
	Since both $ \sigma(x_1) $ and $ \sigma(x_2) $ stabilize $ e_1+e_2 $, it follows that
	\[
	\sigma(x_2)\sigma(x_1)M(e_1+e_2,e_1-e_2)
	\sigma(x_1)^{\sf T} \sigma(x_2)^{\sf T} =
	-M(e_1+e_2,e_3-e_4).  \qedhere
	\]
\end{proof}

It remains to prove \eqref{eq:SufficientConditionIrredEn} for
$ n \in \{6,7\} $.
\begin{proposition}
	We have
	\begin{align*}
		\dim W(E_n) \cdot U_2(D_{s_n})  \ge \dim \mathcal{T}^{s_n}_0
		+n  = \dim \mathcal{T}^n_0,
	\end{align*}
	so $
	W(E_n)\cdot U_2(D_{s_n}) \cong \mathcal{T}^n_0$.
\end{proposition}

\begin{proof}
	We identify $ \mathcal{T}^{s_n}_0 $ with the space of all traceless symmetric matrices whose last $ n-s_n $ rows respectively columns are zeros.
	As a consequence of Lemma~\ref{lemma:TwoOrbitsBecomeOne}, $ \mathcal{T}^{s_n}_0 \subset W(E_n) \cdot U_2(D_{s_n}) $, so it suffices to find $ n $ matrices $ X_1,\ldots,X_n \in (W(E_n) \cdot M(x,y)  )\setminus   \mathcal{T}^{s_n}_0 $ such that
	\begin{align}  \label{eq:adding-elements-to-W(E_n)-orbit}
		\dim \lspan  \{ \mathcal{T}^{s_n}_0, X_1,\ldots,X_n\}
		= \dim \mathcal{T}^{s_n}_0  +n.
	\end{align}
	Therefore, observe that for each root $ z \in E_n \setminus D_{s_n} $ we can find a tuple of roots $ x,y \in D_{s_n} $ and an element $ S \in W(E_n)$ such that
	\begin{align}  \label{eq:LeaveT_{s_n,0}OverWeylGroup}
		Sx = z \quad \text{and}  \quad Sy = y.
	\end{align}
	The action of $ S $ maps $ xx^{\sf T} -yy^{\sf T} \in \mathcal{T}^{s_n}_0 $ to $ zz^{\sf T} - yy^{\sf T} \notin \mathcal{T}^{s_n}_0  $.
	To see \eqref{eq:LeaveT_{s_n,0}OverWeylGroup}, if $ z = \frac{1}{2}(a_1,\ldots,a_8) \in  E_n \setminus D_{s_n} $ with $ a_i \in \{\pm 1\} $, choose
	\begin{align*}
		x &= (a_1,a_2,0,\ldots,0), \quad   y = (a_1,-a_2,0,\ldots,0) \quad \text{and}  \\
		S &= \sigma(z')  \quad \text{with}  \quad z' =\frac{1}{2} (a_1,a_2,-a_3,\ldots,-a_8).
	\end{align*}
	Then, one can directly verify that $ Sx = z $ and $ Sy = y $.

	Now, choose a set of linearly independent roots
	$ z_1,\ldots,z_n \in E_{n} \setminus D_{s_n} $.  Such a set
	exists, for example, take the roots
	$ z_1,\ldots,z_n \in E_n \setminus D_{s_n}$ such that the
	$ i $-th and $ (i+1) $-th entry of root $ z_i $ for
	$ 1 \le i \le n-1 $ are negative and the remaining entries
	positive, and for $ z_n $ we set the first and the $ n $-th
	entry to be negative and the remaining ones positive.

	Additionally, choose $ y_1,\ldots,y_n \in D_{s_n} $. Then the
	matrices $ X_i = z_iz_i^{\sf T} - y_iy_i^{\sf T} $ lie in
	$ W(E_n) \cdot U_2(D_{s_n}) \setminus \mathcal{T}^{s_n}_0 $.  These
	matrices are linearly independent since the last row of
	$ z_iz_i^{\sf T} - y_iy_i^{\sf T} $ is given by the vector
	$ \pm 1/2 z_i $ and vectors $ z_i $ were chosen to be linearly
	independent. Since the last row of every matrix in
	$ \mathcal{T}^{s_n}_0 $ consists of only zeros, it follows
	that adding these vectors to $ \mathcal{T}^{s_n}_0 $ increases
	the dimension of their joint span by $ n $, which proves
	\eqref{eq:adding-elements-to-W(E_n)-orbit}.
\end{proof}

\subsubsection{Proof of Theorem \ref{Theorem:Eigenvalues}}

To prove Theorem~\ref{Theorem:Eigenvalues} it remains to compute
 \[
 Q[A] = \lambda \Tr A^2
 \]
 for $A$ contained in one of the spaces given in Theorem
 \ref{theorem:Peter-Weyl-For-Sn}.

We first evaluate $Q$ at the identity matrix. We have
\[
 Q[I_n] = \sum_{r\in R} (r^{\sf T}r)^2 = 4 |R|,
\]
and using $\Tr I_n = n$ we see that $\lambda = 4h$, where $h$ is the
Coxeter number of the root system $R$.

Note that for $ R = A_n $ or $ R = D_n $ we can find $ x,y \in R $ with $ x \cdot y = 0 $ and $\{x,y\}\neq\{e_i -e_j,e_i +e_j\}$ such that $ M(x,y) \in U_1(R) $. In the case of $ D_4 $ we can find such an element $ M(x,y) $ in both of the two irreducible subspaces decomposing $ U_1(D_4) $.
Then
 \begin{align*}
    Q[M(x,y)] &=\sum_{r\in R} M(x,y)[r]^2 \\
            & = 4\sum_{r\in R}(x\cdot r)^2(y\cdot r)^2.
\end{align*}
We only have to consider roots $r$, with $r\cdot x \neq 0$ and
$r\cdot y \neq 0$, which implies $(r\cdot x)^2 = (r\cdot y)^2 = 1$. For
$R=A_n$ we can find $8$ roots fulfilling this condition, for $R = D_n$
there are $16$.  Hence,
\[
  Q[M(x,y)] =
  \begin{cases}
    32 & \text{for } R = A_n,\\
    64 & \text{for } R = D_n.
  \end{cases}
\]
For the matrices $M(e_i -e_j,e_i +e_j) \in U_2(D_n)$, the result is similarly
\[
Q[M(e_i -e_j,e_i +e_j)] = 4\sum_{r\in D_n}((e_i +e_j)\cdot r)^2((e_i -e_j)\cdot r)^2.
\]
If $r= \pm e_i \pm e_j$, the summand is zero. Otherwise, if
$(r\cdot (e_i +e_j))^2 =1$, it follows $(r\cdot (e_i -e_j))^2=1$, and
there are exactly $8(n-2)$ such roots $r \in D_n$.  Hence,
$Q[M(e_i -e_j,e_i +e_j)] = 32(n-2)$.

In all three cases, the normalizing factor is
\[
\Tr M(x,y)^2 = 2 (x\cdot x) (y\cdot y) + 2(x\cdot y)^2 = 8.
\]
So we obtain eigenvalues $ 4 $ on $ U_1(A_n) $, respectively $ 8 $
and $ 4(n-2) $ on $ U_1(D_n) $ and $ U_2(D_n) $.

For $R = A_n$ we have to compute the eigenvalue for $ U_2(A_n) $, so we
may evaluate $ Q(P_1) $. Observe that
\[
P_1[r]^2=\left( \sum_{j\in\{2,\ldots,n+1\}} ((e_1-e_j)\cdot r)^2-4\right)^2.
\]
If $r=\pm (e_1-e_j)$ for some $ j \in \{2,\ldots,n+1\} $, then we get $(r,r)^2=4$ and $((e_1-e_j)\cdot r)^2=1$ for all other $j$. This amounts to
\[
P_1[r]^2 = \left( 4 + (n-1)- 4 \right)^2 = (n-1)^2.
\]
If $r=(e_k- e_l)$ with $k,l\neq 1$, it follows $(r\cdot (e_1-e_k))^2=(r\cdot (e_1-e_l))^2=1$ and all other summands are zero. So we get
\[
P_1[r]^2 = (2-4)^2 = 4.
\]
There are $2n$ roots of type $ \pm (e_1-e_j) $ and accordingly
$n(n-1)$ of type $ (e_k-e_l) $ with $ k,l\neq 1 $. This results in
\[
Q[P_1] = 2n(n-1)^2+4n(n-1) = 2n(n-1)(n+1).
\]
Now we compute $\Tr P_1^2 = \langle P_1, P_1 \rangle$ and get
\begin{align*}
	 \langle P_1,P_1\rangle  &= \left \langle \sum_{j \in \{2,\ldots,n+1\}} M(e_1-e_j) -2I_n, \, \sum_{j \in \{2,\ldots,n+1\}} M(e_1-e_j) - 2I_n \right \rangle    \\
	&= \sum_{j,k \in \{2,\ldots,n+1\}} \langle M(e_1-e_j),
   M(e_1-e_k) \rangle  - 4 \sum_{j \in \{2,\ldots,n+1\}} \langle
   M(e_1-e_j), I_n \rangle + 4n\\
	&= \sum_{j,k \in \{2,\ldots,n+1\}} ((e_1-e_k)\cdot (e_1-e_l) )^2  - 4n.
\end{align*}
The first sum equals
\begin{align*}
	\sum_{j\in\{2,\ldots,n+1\}} ((e_1-e_j)\cdot (e_1-e_j))^2+ \sum_{2 \le j \neq k \le n+1} ((e_1-e_j)\cdot (e_1-e_k))^2 =4n+n(n-1).
\end{align*}
Hence, together we have $\langle P_1,P_1\rangle = n(n-1)$ and the
eigenvalue associated with the eigenspace $ U_2(A_n) $ is $2(n+1)=2h$.

The remaining eigenvalues for $E_6,E_7,E_8$ are given by the fact that
these root systems form spherical $4$-designs.  Then, by
\eqref{eq:spherical-4-design-quadratic-form},
$Q[H] = \frac{8h}{n+2}\Tr H^2$. \qed

\subsection{Orthogonal sum of irreducible root systems}

In this section, we want to compute the eigenvalues of the quadratic
form $ Q $ on the orthogonal sum of irreducible root systems
$ R=R_1\perp\ldots \perp R_m$. For this we write
\[
Q_{R_i}[H] = \sum_{x\in R_i} H[x]^2
\]
to distinguish between the quadratic form on different root systems
$R_i$. Let $n_i$ be the rank of $R_i$. Furthermore, let
$n=n_1+\cdots + n_m$. Write each $x\in \R^n$ as $(x_1,\ldots,x_m)$
with $x_i\in \R^{n_i}$ and every root $r\in R$ as
$(0,\ldots,0,r_i,0,\ldots,0)$ with $r_i\in R_i$ and $0\in \R^{n_j}$
accordingly. To compute the eigenvalues of $Q_R$ on $\mathcal{S}^n$,
we identify $\mathcal{S}^n$ in a similar fashion: Each
$H\in \mathcal{S}^n$ can be seen as a vector of block matrices

\begin{equation}\label{eq:identification-S^n-matrix}
H \cong  (H_{1,1},\ldots,H_{m,m},H_{1,2},\ldots,H_{m-1,m}) \; \Longleftrightarrow \; H = \begin{pmatrix}
H_{1,1} & H_{1,2} & \cdots & H_{1,m} \\
H_{1,2}^{\sf T} & H_{2,2} & \cdots & H_{2,m} \\
\vdots & \vdots & \ddots & \vdots \\
H_{1,m}^{\sf T} & H_{2,m}^{\sf T} & \cdots & H_{m,m}
\end{pmatrix} ,
\end{equation}
where $H_{i,i}\in \mathcal{S}^{n_i}$ and $H_{i,j}\in \R^{n_i\times n_j}$ for $i\neq j$. This way, we identify
\begin{equation}\label{eq:identification-S^n}
\mathcal{S}^n \cong \mathcal{S}^{n_1}\perp \ldots \perp \mathcal{S}^{n_m}\perp \bigperp_{1\leq i<j\leq m} \R^{n_i\times n_j}.
\end{equation}
Furthermore, let $ \mathcal{D} $ be the $ m $-dimensional space that is spanned by the diagonal matrices
\[
(I_{n_1},0,\ldots,0), (0,I_{n_2},0,\ldots,0), \ldots, (0,\ldots,0,I_{n_m},0,\ldots,0).
\]

We are particularly interested in the case where each component of the
root system $ R $ has the same Coxeter number. In this case $ R $ is
of the form
\begin{align}\label{eq:same-coxeter-number}
	R =  (A_{n_a})^{m_a} \perp (D_{n_d})^{m_d}  \perp (E_{n_e})^{m_e},
\end{align}
where $ (A_{n_a})^{m_a}, \, (D_{n_d})^{m_d} $ respectively
$ (E_{n_e})^{m_e} $ are orthogonal sums of $ m_a, m_d$ respectively
$ m_e $ irreducible roots systems $ A_{n_a}, D_{n_d}$ respectively
$E_{n_e} $, and $m=m_a+m_d+m_e$, $n=m_an_a+m_dn_d+m_en_e$.

\begin{theorem}\label{theorem:orthogonal sum}
  Let $R=\bigperp_{i=1}^m R_i$ be the orthogonal sum of irreducible
  root systems $R_i\in\{A_{n_i},D_{n_i},E_{n_i }\}$, where $n_i$ is
  the rank of $R_i$. We identify $\mathcal{S}^n$ as in
  \eqref{eq:identification-S^n}.
\begin{enumerate}[label=(\roman*)]

\item \label{item:composite-root-system-i}
	We have
	\begin{equation}\label{eq:split-quadratic-form}
	Q_R[H] = Q_{R_1}[H_{1,1}] + \cdots + Q_{R_m}[H_{m,m}],
	\end{equation}
	so the quadratic form only depends on the diagonal entries
        $H_{i,i}\in \mathcal{S}^{n_i}$ and the eigenvalues of $Q_R$
        are the eigenvalues of all $Q_{R_i}$ and additionally the
        eigenvalue $0$ with multiplicity
        $\sum_{1\leq i<j\leq m} n_i n_j$.

      \item\label{item:composite-root-system-ii} If each component
        root system has the same Coxeter number $h$, we can write $R$
        as in \eqref{eq:same-coxeter-number}. The space of traceless
        matrices $\Tcal_0^n$ then decomposes into eigenspaces of
        $Q_R$:
\begin{equation}
  \label{eq:split-traceless}
  \begin{split}
\Tcal_0^n \; =  \; & \phantom{\perp} \;\; U_1(A_{n_a})^{m_a} \perp U_2(A_{n_a})^{m_a}\\
	 & \perp U_1(D_{n_d}) ^{m_d}\perp U_2(D_{n_d}) ^{m_d}   \\
& \perp  (\Tcal_0^{n_e})^{m_e}\\
 &\perp \mathcal{D}\cap \Tcal_0^n,
      \end{split}
\end{equation}
where the exponents refer to the direct sum of the eigenspaces of
$ Q_{A_{n_a}}, Q_{D_{n_d}} $ and $ Q_{E_{n_e}} $.  The eigenspace
$\mathcal{D} \cap \Tcal_0^n$ belongs to the eigenvalue $4h$ and has
dimension $m-1$.
\end{enumerate}
\end{theorem}

\begin{remark}
  The decomposition \eqref{eq:split-traceless} does not change when
  $ D_4 $ is considered because the quadratic form has the same
  eigenvalues on both irreducible subspaces that decompose
  $ U_1(D_4) $.
\end{remark}

\begin{proof}
  \ref{item:composite-root-system-i} Let $H\in \mathcal{S}^n$. We
  write $H$ as in \eqref{eq:identification-S^n-matrix}. For a root
  $r=(0,\ldots,0,r_i,0,\ldots,0)\in R$ it follows
\[
H[r] = H_{i,i}[r_i],
\]
so $H[r]$ does not depend of the off-diagonal entries $(0,\ldots,0,H_{i,j},0,\ldots,0)$ for $i\neq j$ of $H$.

Since every root in $R$ is of this form, this directly implies
\eqref{eq:split-quadratic-form}. This also shows that the eigenvalues
of $Q_R$ coincide with the eigenvalues of $Q_{R_i}$ with the same
multiplicity. The only additional eigenvalue we get is $0$, which is
obtained by evaluating $ Q_R[H] $ for matrices $H\in\mathcal{S}^n$,
where all diagonal entries $H_{ii}=0\in \mathcal{S}^{n_i}$. Due to the
identification \eqref{eq:identification-S^n}, the space of these
matrices has dimension $\sum_{1\leq i<j\leq m} n_i n_j$, which gives
the multiplicity of the eigenvalue $ 0 $.

\ref{item:composite-root-system-ii} If each component root system of
$R$ has the same Coxeter number, the space $\mathcal{D} $ is an
eigenspace of $ Q_R $ because
\begin{align*}
	Q_R\Big [(0,\ldots,0,I_{n_a},0\ldots,0) \Big]&= Q_{A_{n_a}}[I_{n_a}] = 4h,  \\
	Q_R\Big [(0,\ldots,0,I_{n_d},0\ldots,0) \Big]&= Q_{D_{n_d}}[I_{n_d}] = 4h,   \\
	Q_R\Big[(0,\ldots,0,I_{n_e},0\ldots,0) \Big]&= Q_{E_{n_e}}[I_{n_e}] = 4h.
\end{align*}
Hence, $ \mathcal{S}^n $ decomposes into eigenspaces of $ Q_R $ as
\[
\mathcal{S}^n = U_1(A_{n_a})^{m_a} \perp U_2(A_{n_a})^{m_a}
\perp U_1(D_{n_d}) ^{m_d}\perp U_2(D_{n_d}) ^{m_d}
\perp (\Tcal^{n_e}_0)^{m_e}
\perp  \mathcal{D}.
\]
        All eigenspaces but $\mathcal{D}$ lie in the space
        $ \mathcal{T}^{n}_0 $, hence equation
        \eqref{eq:split-traceless} holds.  To see that
        $\mathcal{D} \cap \mathcal{T}^n_0$ has dimension $m-1$, note
        that it contains all diagonal matrices of the form
\[
(c_1I_{n_1},\ldots,c_m I_{n_m},0,\ldots,0)\quad \text{with} \quad c_1n_1+\cdots + c_mn_m = 0.
\]
Since $\mathcal{D}$ has dimension $m$, it follows that $ \mathcal{D} \cap \mathcal{T}^n_0 $ has dimension $ m-1 $.
\end{proof}

\section{Concrete results}
\label{sec:Results}

\subsection{Dimension 8}
\label{ssec:Dimension-8}

Mordell \cite{Mordell1938a} showed that the root lattice $E_8$ is the
only even unimodular lattice in dimension $8$. By \cite{Cohn2019a}
$E_8$ is universally optimal and unique among periodic point
configurations. The fact that it is a local minimum for all Gaussian
potential functions was established in
\cite{Sarnak2006a}. Coulangeon~\cite{Coulangeon2006a}
used~\eqref{eq:Coulangeon-Hessian} to provide an alternative proof.

\subsection{Dimension 16}
\label{ssec:Dimension-16}

Witt \cite{Witt1941a} proved that there exist exactly two even
unimodular lattices in dimension $16$: $D^+_{16}$ and $E_8 \perp
E_8$. Both lattices have the same theta series $E_4^2$, but their root
systems differ as we have $D^+_{16}(2) = D_{16}$ and, respectively,
$E_8 \perp E_8(2) = E_8 \perp E_8$. The eigenvalues of the quadratic
form~\eqref{eq:crucial-quadratic-form} are by
Theorem~\ref{Theorem:Eigenvalues} and
Theorem~\ref{theorem:orthogonal sum}
\[
  8\; (120\times) , 56 \; (15 \times)\quad \text{respectively} \quad 0
  \; (64 \times) , 24 \; (70 \times), 120 \; (1 \times).
\]
Therefore, by Corollary~\ref{cor:large-alpha}, $D^+_{16}$ is a local
minimum for $f_\alpha$-potential energy whenever $\alpha$ is large
enough. By Corollary~\ref{cor:large-alpha} the other lattice
$E_8 \perp E_8$ is a saddle point whenever $\alpha$ is large
enough. The following numerical computations strongly suggest that
$E_8 \perp E_8$ is in fact a saddle point for all values of $\alpha$.

Using SageMath \cite{SageMath} we arrive at the following plot for the
eigenvalues of the Hessian of the function
$L \mapsto \mathcal{E}(f_\alpha, L)$ at $D_{16}^+$ and at
$E_8 \perp E_8$.

\begin{figure}[htb]
\begin{center}
\includegraphics[width=6.2cm]{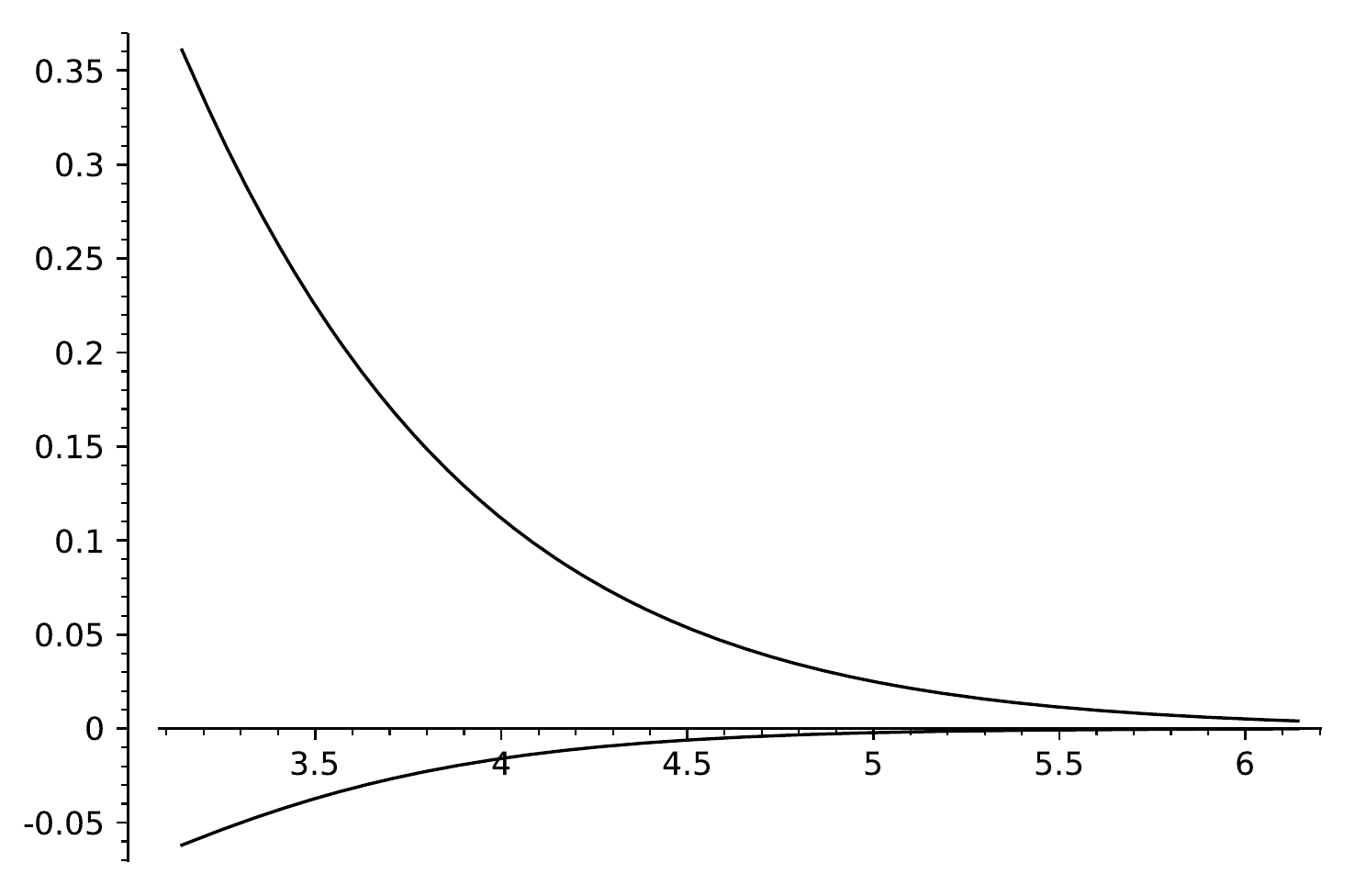}
\includegraphics[width=6.2cm]{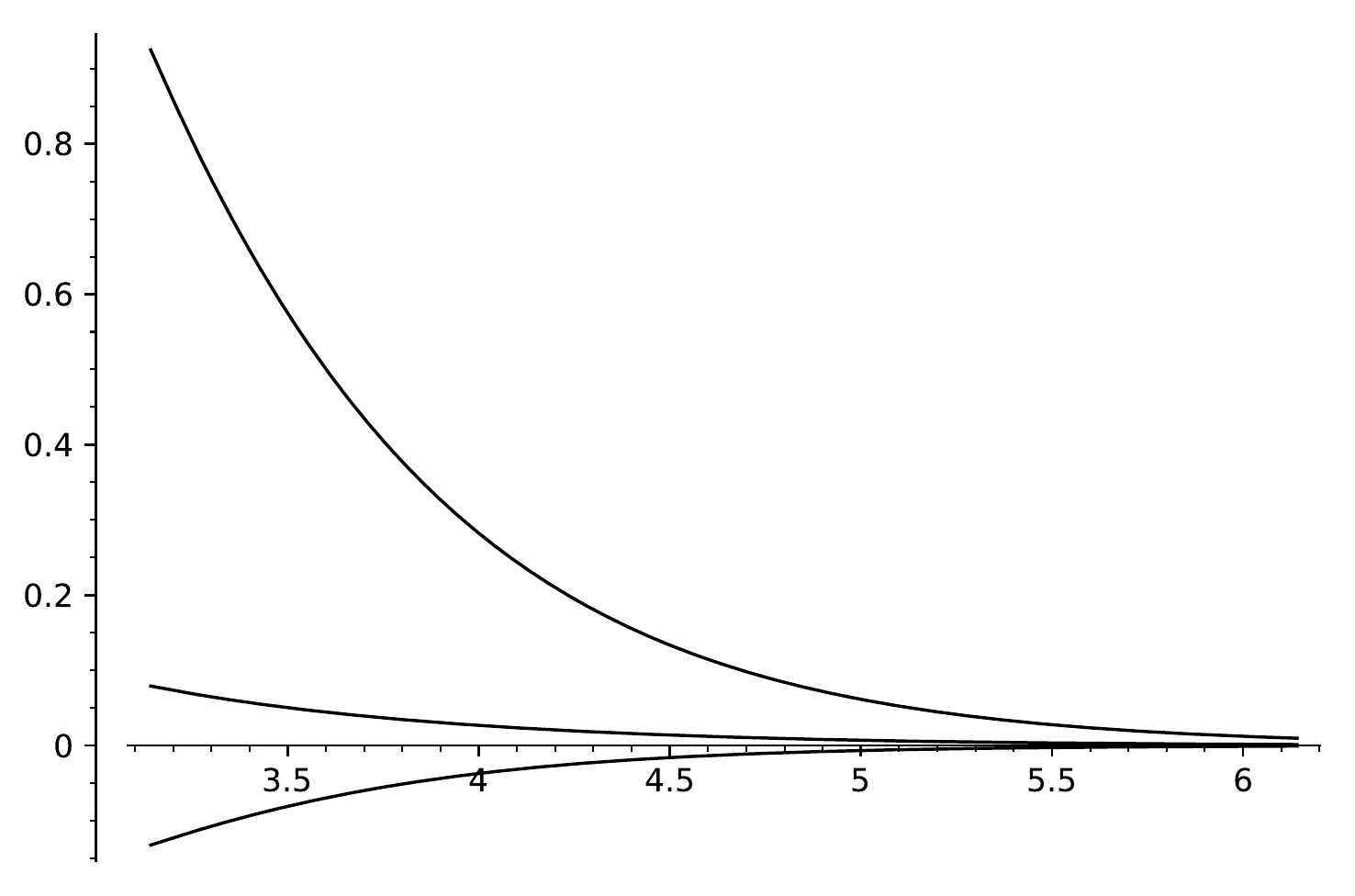}
\end{center}
\caption{The eigenvalues of the Hessian for $D_{16}^+$ (two different
  eigenvalues, left) and $E_8 \perp E_8$ (three different eigenvalues,
  right) depending on the parameter $\alpha$.}
\label{figure:Dimension-16}
\end{figure}

We introduce the following notation: The value in~\eqref{eq:main} we
denote by $\mu(L,\lambda,\alpha)$. We consider $\alpha = \pi$, then
\[
  \begin{array}{lcrlcr}
    \mu(D_{16}^+,8,\pi) & = & -0.06196\ldots
    & \qquad \mu(E_8 \perp E_8, 0,\pi) & = & -0.13245\ldots \\
    \mu(D_{16}^+,56,\pi) & = & 0.36093\ldots
    & \qquad \mu(E_8 \perp E_8, 24,\pi) & = & 0.07899\ldots\\
& &  & \qquad \mu(E_8 \perp E_8, 120,\pi) & = & 0.92480\ldots
   \end{array}
\]
We show in Section~\ref{ssec:Dimension-32} how to translate numerical
computations into rigorous bounds.

\subsection{Dimension 24}
\label{ssec:Dimension-24}

The Niemeier lattices are the even unimodular lattices in dimension
$24$ which have vectors of squared norm $2$. A classification of
Niemeier gave that there are $23$ Niemeier lattices and Venkov
realized that they can be characterized by their root system. The
theta series of a Niemeier lattice $L$ with root system $L(2)$ is the
modular form of weight $12$
\[
\Theta_{L}(\tau) = E_4^3(\tau) + (|L(2)| - 720) \Delta(\tau) = 1 + |L(2)| q + \cdots.
\]
The cusp form of weight $16$ is $E_4 \Delta$. We apply
Theorem~\ref{thm:main} together with Theorem~\ref{Theorem:Eigenvalues}
and Theorem~\ref{theorem:orthogonal sum} to determine the signature of
the Hessian at $\alpha = \pi$. We collect our results in
Table~\ref{tab:Dimension-24}. For large values of $\alpha$
Corollary~\ref{cor:large-alpha} shows that only the Niemeier lattices
with irreducible root systems, namely $A_{24}$ and $D_{24}$, are local
minima for $f_\alpha$-potential energy. All other Niemeier lattices
are saddle points for $f_\alpha$-potential energy for $\alpha$ large enough.

\newpage

\begin{small}
\begin{table}[H]
\begin{tabular}{@{}cccccr@{}}
\toprule
$\bm{L(2)}$ & $\bm{|L(2)|}$ & $\bm{h}$ & $\bm{\lambda}$ & \textbf{multiplicity} &
                                                   $\bm{\mu(L,\lambda,\pi)}$
  \\
  \midrule

    \multirow{2}{*}{$A_1^{24}$} &
   \multirow{2}{*}{$48$} &
   \multirow{2}{*}{$2$} & $0$ & $276$ & $0.0018\ldots$\\
    & & & $8$ & $23$ & $0.1044\ldots$\\
  \midrule

    \multirow{3}{*}{$A_2^{12}$} &
    \multirow{3}{*}{$72$} &
    \multirow{3}{*}{$3$} &
                                       $0$ & $264$ &
                                                     $-0.0050\ldots$\\
     & & & $6$ & $24$ & $0.0718\ldots$\\
            & & & $12$ & $11$ & $0.1488\ldots$\\
  \midrule

\multirow{3}{*}{$A_3^8$} &
\multirow{3}{*}{$96$} &
\multirow{3}{*}{$4$} &
$0$ &
$252$ &
$-0.0120\ldots$\\
& & &
$4$ &
$16$ &
       $0.0392\ldots$\\
& & &
  $8$ &
$24$ &
$0.0905\ldots$\\
& & &
$16$ &
$7$ &
$0.1931\ldots$\\
\midrule

\multirow{4}{*}{$A_4^6$} &
\multirow{4}{*}{$120$} &
\multirow{4}{*}{$5$} &
$0$ &
$240$ &
$-0.0189\ldots$\\
& & &
$4$ &
$30$ &
$0.0323\ldots$\\
& & &
$10$ &
$24$ &
$0.1092\ldots$\\
& & &
$20$ &
$5$ &
$0.2375\ldots$\\
\midrule

\multirow{5}{*}{$A_5^4 D_4$} &
\multirow{5}{*}{$144$} &
\multirow{5}{*}{$6$} &
$0$ &
$230$ &
$-0.0259\ldots$\\
& & &
$4$ &
$36$ &
$0.0253\ldots$\\
& & &
$8$ &
$9$ &
$0.0766\ldots$\\
& & &
$12$ &
$20$ &
$0.1279\ldots$\\
& & &
$24$ &
$4$ &
$0.2818\ldots$\\
  \midrule

\multirow{3}{*}{$D_4^6$} &
\multirow{3}{*}{$144$} &
\multirow{3}{*}{$6$} &
$0$ &
$240$ &
$-0.0259\ldots$\\
& & &
$8$ &
$54$ &
$0.0766\ldots$\\
& & &
$24$ &
$5$ &
$0.2818\ldots$\\
  \midrule

\multirow{4}{*}{$A_6^4$} &
\multirow{4}{*}{$168$} &
\multirow{4}{*}{$7$} &
$0$ &
$216$ &
$-0.0328\ldots$\\
& & &
$4$ &
$56$ &
$0.0184\ldots$\\
& & &
$14$ &
$24$ &
$0.1466\ldots$\\
& & &
$28$ &
$3$ &
$0.3262\ldots$\\
  \midrule

\multirow{5}{*}{$A_7^2 D_5^2$} &
\multirow{5}{*}{$192$} &
\multirow{5}{*}{$8$} &
$0$ &
$214$ &
$-0.0398\ldots$\\
& & &
$4$ &
$40$ &
$0.0114\ldots$\\
& & &
$8$ &
$20$ &
$0.0627\ldots$\\
& & &
$12$ &
$8$ &
      $0.1140\ldots$\\
  & & &
$16$ &
$14$ &
$0.1653\ldots$\\
& & &
$32$ &
$3$ &
$0.3705\ldots$\\
  \midrule

\multirow{4}{*}{$A_8^3$} &
\multirow{4}{*}{$216$} &
\multirow{4}{*}{$9$} &
$0$ &
$192$ &
$-0.0467\ldots$\\
& & &
$4$ &
$81$ &
$0.0045\ldots$\\
& & &
$18$ &
$24$ &
$0.1840\ldots$\\
& & &
$36$ &
$2$ &
$0.4149\ldots$\\
  \midrule

\multirow{6}{*}{$A_9^2D_6$} &
\multirow{6}{*}{$240$} &
\multirow{6}{*}{$10$} &
$0$ &
$189$ &
$-0.0537\ldots$\\
& & &
$4$ &
$70$ &
$-0.0024\ldots$\\
& & &
$8$ &
$15$ &
$0.0488\ldots$\\
& & &
$16$ &
$5$ &
      $0.1514\ldots$\\
  & & &
$20$ &
$18$ &
$0.2027\ldots$\\
& & &
$40$ &
$2$ &
$0.4592\ldots$\\
  \midrule

\multirow{4}{*}{$D_6^4$} &
\multirow{4}{*}{$240$} &
\multirow{4}{*}{$10$} &
$0$ &
$216$ &
$-0.0537\ldots$\\
& & &
$8$ &
$60$ &
$0.0488\ldots$\\
& & &
$16$ &
$20$ &
$0.1514\ldots$\\
& & &
$40$ &
$3$ &
$0.4592\ldots$\\
  \midrule

\multirow{3}{*}{$E_6^4$} &
\multirow{3}{*}{$288$} &
\multirow{3}{*}{$12$} &
$0$ &
$216$ &
$-0.0676\ldots$\\
& & &
$12$ &
$80$ &
$0.0862\ldots$\\
& & &
$48$ &
$3$ &
$0.5479\ldots$\\
\bottomrule
\end{tabular}
\bigskip
\caption{The eigenvalues of the Hessian of the Niemeier lattices for
  $\alpha = \pi$.}
\label{tab:Dimension-24}
\end{table}
\end{small}

\begin{small}
\begin{table}[H]
\begin{tabular}{@{}cccccr@{}}
\toprule
$\bm{L(2)}$ & $\bm{|L(2)|}$ & $\bm{h}$ & $\bm{\lambda}$ & \textbf{multiplicity} &
                                                   $\bm{\mu(L,\lambda,\pi)}$ \\    \midrule
\multirow{6}{*}{$A_{11} D_7 E_6$} &
\multirow{6}{*}{$288$} &
\multirow{6}{*}{$12$} &
$0$ &
$185$ &
$-0.0676\ldots$\\
& & &
$4$ &
$54$ &
$-0.0163\ldots$\\
& & &
$8$ &
$21$ &
$0.0349\ldots$\\
& & &
$12$ &
$20$ &
$0.0862\ldots$\\
& & &
$20$ &
$6$ &
      $0.1888\ldots$\\
& & &
$24$ &
$11$ &
$0.2401\ldots$\\
& & &
$48$ &
$2$ &
$0.5479\ldots$\\
  \midrule

\multirow{4}{*}{$A_{12}^2$} &
\multirow{4}{*}{$312$} &
\multirow{4}{*}{$13$} &
$0$ &
$144$ &
$-0.0746\ldots$\\
& & &
$4$ &
$130$ &
$-0.0233\ldots$\\
& & &
$26$ &
$24$ &
$0.2588\ldots$\\
& & &
$52$ &
$1$ &
$0.5923\ldots$\\
  \midrule

\multirow{4}{*}{$D_8^3$} &
\multirow{4}{*}{$336$} &
\multirow{4}{*}{$14$} &
$0$ &
$192$ &
$-0.0815\ldots$\\
& & &
$8$ &
$84$ &
$0.0210\ldots$\\
& & &
$24$ &
$21$ &
$0.2262\ldots$\\
& & &
$56$ &
$2$ &
$0.6366\ldots$\\
  \midrule

\multirow{6}{*}{$A_{15}D_9$} &
\multirow{6}{*}{$384$} &
\multirow{6}{*}{$16$} &
$0$ &
$135$ &
$-0.0954\ldots$\\
& & &
$4$ &
$104$ &
$-0.0441\ldots$\\
& & &
$8$ &
$36$ &
$0.0071\ldots$\\
& & &
$28$ &
$8$ &
      $0.2636\ldots$\\
& & &
$32$ &
$15$ &
$0.3149\ldots$\\
& & &
$64$ &
$1$ &
$0.7253\ldots$\\
  \midrule

\multirow{5}{*}{$A_{17} E_7$} &
\multirow{5}{*}{$432$} &
\multirow{5}{*}{$18$} &
$0$ &
$119$ &
$-0.1093\ldots$\\
& & &
$4$ &
$135$ &
$-0.0580\ldots$\\
& & &
$16$ &
$27$ &
$0.0958\ldots$\\
& & &
$36$ &
$17$ &
$0.3523\ldots$\\
& & &
$72$ &
$1$ &
$0.8140\ldots$\\
  \midrule

\multirow{5}{*}{$D_{10} E_7^2$} &
\multirow{5}{*}{$432$} &
\multirow{5}{*}{$18$} &
$0$ &
$189$ &
$-0.1093\ldots$\\
& & &
$8$ &
$45$ &
$-0.0067\ldots$\\
& & &
$16$ &
$54$ &
$0.0958\ldots$\\
& & &
$32$ &
$9$ &
$0.3010\ldots$\\
& & &
$72$ &
$2$ &
$0.8140\ldots$\\
  \midrule

\multirow{4}{*}{$D_{12}^2$} &
\multirow{4}{*}{$528$} &
\multirow{4}{*}{$22$} &
$0$ &
$144$ &
$-0.1371\ldots$\\
& & &
$8$ &
$132$ &
$-0.0345\ldots$\\
& & &
$40$ &
$22$ &
$0.3758\ldots$\\
& & &
$88$ &
$1$ &
$0.9914\ldots$\\
  \midrule

\multirow{2}{*}{$A_{24}$} &
\multirow{2}{*}{$600$} &
\multirow{2}{*}{$25$} &
$4$ &
$275$ &
$-0.1067\ldots$\\
& & &
$50$ &
$24$ &
$0.4832\ldots$\\
  \midrule

\multirow{5}{*}{$D_{16}E_8$} &
\multirow{5}{*}{$720$} &
\multirow{5}{*}{$30$} &
$0$ &
$128$ &
$-0.1928\ldots$\\
& & &
$8$ &
$120$ &
$-0.0902\ldots$\\
& & &
$24$ &
$35$ &
$0.1150\ldots$\\
& & &
$56$ &
$15$ &
$0.5254\ldots$\\
& & &
$120$ &
$1$ &
$1.3462\ldots$\\
  \midrule

\multirow{3}{*}{$E_8^3$} &
\multirow{3}{*}{$720$} &
\multirow{3}{*}{$30$} &
$0$ &
$192$ &
$-0.1928\ldots$\\
& & &
$24$ &
$105$ &
$0.1150\ldots$\\
& & &
$120$ &
$2$ &
$1.3462\ldots$\\
\midrule
\multirow{2}{*}{$D_{24}$} &
\multirow{2}{*}{$1104$} &
\multirow{2}{*}{$46$} &
$8$ &
$276$ &
$-0.2014\ldots$\\
& & &
$88$ &
$23$ &
$0.8246\ldots$\\
 \bottomrule
  \end{tabular}
\bigskip
\\
\textsc{Table~\ref{tab:Dimension-24}.} (continued).
\end{table}
\end{small}

\subsection{Dimension 32}
\label{ssec:Dimension-32}

In dimension 32 the even unimodular lattice have not been classified
yet. Some partial results are known: There are at least 80 million of
them, see Serre \cite{Serre1973a}. King \cite{King2002a} showed that
there are at least ten million even unimodular lattices without roots in
dimension 32.  Kervaire \cite{Kervaire1994a} classified all
indecomposable even unimodular lattices in dimension $32$ that possess
a full root system.

In general an even unimodular lattice need not even be a critical
point for the Gaussian potential function. The first such examples can
be found in dimension $32$, we briefly discuss one of them.

For example there exists a lattice $L \subseteq \R^{32}$ with complete
root system $A_1^8A_3^8$, see Kervaire~\cite{Kervaire1994a}. We split
the summation in the gradient into the contribution of the root system
and the contribution of all larger shells
\[
\begin{split}
 \langle \nabla  \mathcal{E}(f_\alpha,L),H\rangle &= -\alpha \sum_{x \in L \setminus \{0\}} H[x] e^{-\alpha \|x\|^2} \\
&= -\alpha e^{-2\alpha} \left(\sum_{x \in L(2)} H[x] \right)  - \alpha \left( \sum_{x \in L \setminus (\{0\} \cup L(2))} H[x] e^{-\alpha \|x\|^2}\right).
\end{split}
\]
We firstly evaluate
\[
\sum_{x \in L(2)} H[x] = \langle H, \sum_{x \in L(2)} x x^{\sf T} \rangle
\]
and use the fact that $A_1$ and $A_3$ form spherical $2$-designs and so
\[
\sum_{x \in L(2)} x x^{\sf T} = 2 h(A_1)  I_8 \oplus 2 h(A_3)  I_{24} = 4 I_8 \oplus 8 I_{24}.
\]
The matrix $H = 24 I_8 \oplus (-8)  I_{24}$ has trace zero and gives
\[
\sum_{x \in L(2)} H[x] = 24 \cdot 4 \cdot 8 - 8 \cdot 8 \cdot 24 = -4 \cdot 8 \cdot 24 = -768 \neq 0.
\]
Now, by the eigenvalue bounds for $H[x]$ coming from the Rayleigh-Ritz
principle, we find that
\[
  -8 = \lambda_{\min}(H) \leq \frac{H[x]}{\|x\|^2} \leq
  \lambda_{\max}(H) = 24.
\]

This allows to organize summation over all lattice vectors of squared length at least $4$ by shells
\[ -8 \sum_{m \geq 2} a_m \cdot 2m \cdot e^{-\alpha (2m)} \leq \sum_{x \in L \setminus (\{0\} \cup L(2))} H[x] e^{-\alpha \|x\|^2} \leq 24 \sum_{m \geq 2} a_m \cdot 2m \cdot   e^{-\alpha (2m)},
\]
where $a_m = |L(2m)|$ ist the $m$-th coefficient of the theta series $\Theta_L$ of $L$.

Combining the above, we see that it suffices to show
\begin{equation} \label{eq:noncritical_final_ineq}
  24 \sum_{m \geq 2} a_m \cdot 2m \cdot e^{-\alpha (2m)} \leq 768 \cdot e^{-2\alpha}.
\end{equation}
For this we write $\Theta_L$ in the form
$\Theta_L = E_{16} + f$, where $f$ is a cusp form of weight $16$. Let
 \[
   E_{16}(\tau) = \sum_{m=0}^\infty b_m q^m \quad \text{and} \quad
   f(\tau) = \sum_{m=1}^\infty c_m q^m,
\]
in particular
$b_m = -\frac{32}{B_{16}} \sigma_{15}(m) = 16320/3617 \sigma_{15}(m)$
and so $c_1 = -16320/3617$. We use the estimate
$\sigma_{k-1}(m) \leq \zeta(k-1) m^{k-1}$, where $\zeta$ is the
Riemann zeta function, and get $b_m \leq 4.6 m^{15}$. To bound $c_m$
we use \eqref{eq:explicit-bound}, the facts $\ell = 1$, $d(m) \leq
2 \sqrt{m}$, and get
$|c_m| \leq 1.2 \cdot 10^{10} m^8$. Together,
\begin{equation} \label{eq:theta_coefficient_bound}
  |a_m| \leq 4.6 m^{15} + 1.2 \cdot 10^{10} m^8.
\end{equation}
We evaluate for $\alpha = 14$, this gives
\[
  \sum_{m \geq 2} a_m \cdot 2m \cdot e^{-28m} \leq 9.2 \sum_{m = 2}^\infty m^{16} \cdot e^{-28m} + 2.4 \cdot 10^{10} \sum_{m = 2}^\infty m^{9} \cdot e^{-28m}.
\]
By Lemma~\ref{lem:integral-test-estimate} we have
\[
  \sum_{m=2}^\infty m^{16}
  e^{-28 m} \leq 3.2 \cdot 10^{-20} + (28)^{-17}
  \Gamma(17, 56) \leq 3.3 \cdot 10^{-20},
\]
and
\[
  \sum_{m=2}^\infty m^{9}
  e^{-28 m} \leq 2.5 \cdot 10^{-22} + (2\alpha)^{-10}
  \Gamma(10, 56)   \leq 2.6 \cdot 10^{-22}.
\]
Putting everything together for $\alpha = 14$ in
\eqref{eq:noncritical_final_ineq} we find
\[
  \begin{split}
  24 \sum_{m \geq 2} a_m \cdot 2m \cdot e^{-\alpha (2m)} \; \leq \; &  24 \left( 9.2 \cdot 3.3 \cdot 10^{-20} + 2.4 \cdot 10^{10} \cdot 2.6 \cdot 10^{-22} \right) \\
  \leq \; & 24 \left( 3.1 \cdot 10^{-19} + 6.3 \cdot 10^{-12} \right) \\
  \leq \; & 1.6 \cdot 10^{-10} \\
  \leq \; &  768 \cdot e^{- 28}.
  \end{split}
\]
This shows that this lattice is not a critical point for the Gaussian
potential function $e^{-14r}$.

\bigskip

Last, but not least, we show that all even unimodular
lattices without roots in dimension $32$ are local maxima for the Gaussian potential
function $e^{-\pi r}$. All the even unimodular lattices in
dimension 32 without roots have the same theta series, for such a lattice
$L \subseteq \R^{32}$ we have
\[
  \begin{split}
    \Theta_L(\tau) \; = \; & E_4^4(\tau) - 960 E_4(\tau) \Delta(\tau)\\
    \; = \; &
    1 + 146880 q^2 + 64757760 q^3 + 4844836800 q^4 + 137695887360 q^5\\
    & \quad + 2121555283200 q^6 + 21421110804480 q^7\\
    & \quad + 158757684004800 q^8 + \cdots
   \end{split}
 \]
 All shells of $L$ form spherical 4-designs, so $L$ is critical for
 all Gaussian potential functions and we can compute the eigenvalue of
 the Hessian \eqref{eq:Hessian} using
 \eqref{eq:Coulangeon-Hessian}. For $\alpha = \pi$ we compute the
 first summands of the series and get
 \[
\frac{1}{n(n+2)} \sum_{m=0}^8  a_m \pi (2m) (\pi (2m) - (n/2+1))
  e^{-\pi(2m)} < -0.00027.
 \]

\begin{figure}[htb]
\begin{center}
\includegraphics[width=8cm]{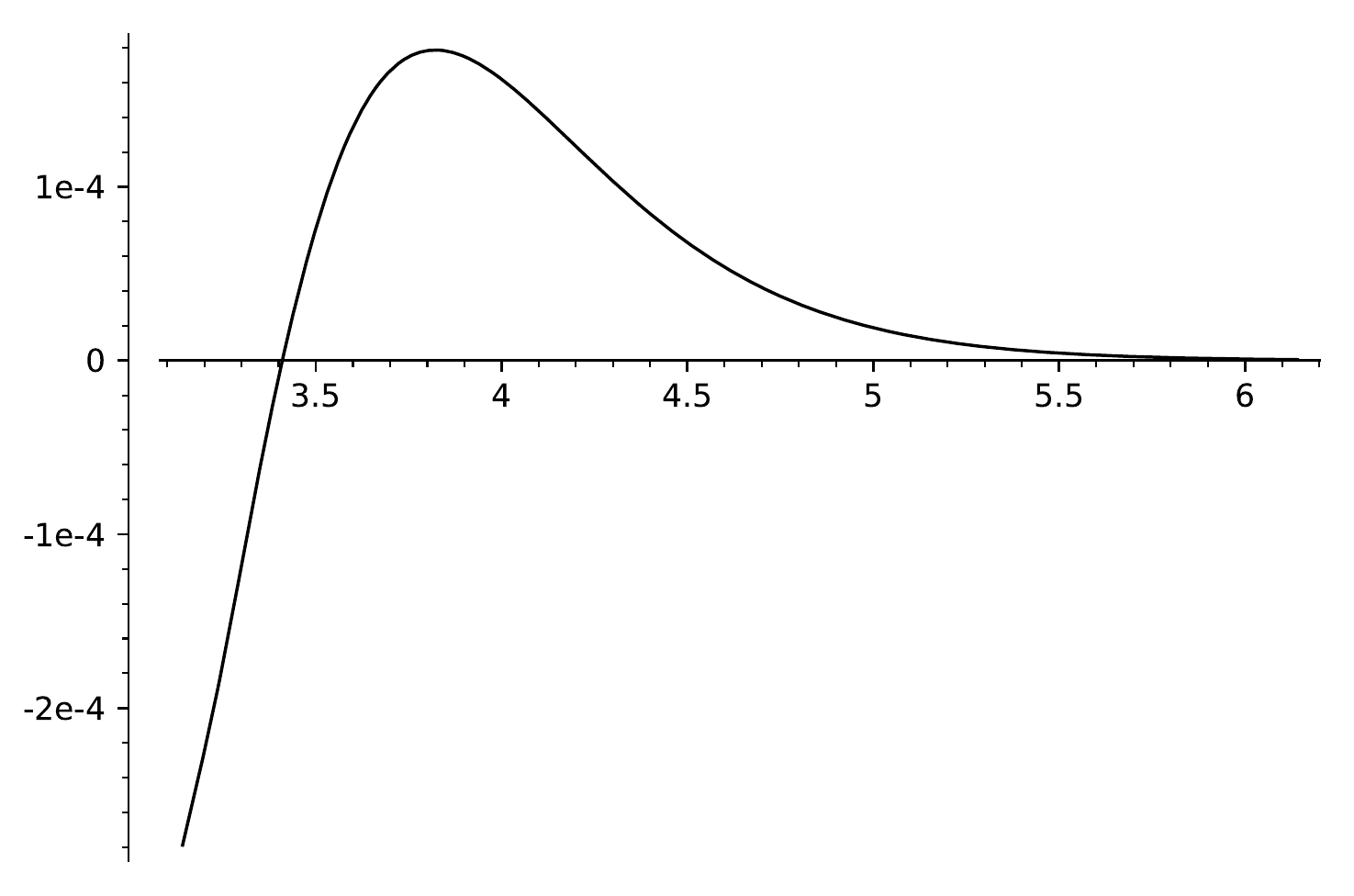}
\end{center}
\caption{The eigenvalue of the Hessian for even unimodular
  lattices in dimension $32$ without roots depending on the parameter $\alpha$.}
\label{figure:Dimension-32}
\end{figure}

Now we argue that the tail of the series is so small that the entire
series is strictly negative.


For this, again, we use the bound \eqref{eq:theta_coefficient_bound}
for the coefficients $a_m$ of $\Theta_L$, and we estimate
\[
\left| \sum_{m=9}^\infty a_m \pi(2m)(\pi(2m) - (n/2+1))
  e^{-\pi(2m)} \right| \leq \sum_{m=9}^\infty |a_m| (2\pi m)^2
  e^{-2\pi m},
\]
and
\[
  \sum_{m=9}^\infty |a_m| (2\pi m)^2
  e^{-2\pi m} \leq   181.7 \sum_{m=9}^\infty m^{17}
  e^{-2\pi m} + 4.8 \cdot 10^{11}  \sum_{m=9}^\infty m^{10}
  e^{-2\pi m}.
\]
Again, by Lemma~\ref{lem:integral-test-estimate}
\[
  \sum_{m=9}^\infty m^{17}
  e^{-2\pi m} \leq 4.7 \cdot 10^{-9} + (2\pi)^{-18}
  \Gamma(18, 18\pi)
  \leq 5.8 \cdot 10^{-9}.
\]
Similarly,
\[
  \sum_{m=9}^\infty m^{10}
  e^{-2\pi m} \leq 9.7 \cdot 10^{-16} + (2\pi)^{-11}
  \Gamma(11, 18\pi) \leq
1.2 \cdot 10^{-15}.
\]
Altogether:
\[
  \begin{split}
& \left|  \frac{1}{n(n+2)} \sum_{m=9}^\infty  a_m \pi (2m) (\pi (2m) - (n/2+1))
  e^{-\pi(2m)} \right|\\
& \leq \;  1088^{-1} \left( 181.7 \cdot 5.8 \cdot
  10^{-9} + 4.8 \cdot 10^{11} \cdot 1.2 \cdot 10^{-15} \right) \\
& \leq \; 5.4 \cdot 10^{-7}.
\end{split}
\]

Hence, we showed that for $\alpha = \pi$ the even unimodular lattices
in dimension $32$ without roots are local maxima for the Gaussian
potential function. This answers a question of Regev and
Stephens-Davidowitz \cite{Regev2016a}.

\section*{Acknowledgements}

F.V. and M.C.Z. thank Noah Stephens-Davidowitz for a discussion at the
Simons Institute for the Theory of Computing during the workshop
``Lattices: Geometry, Algorithms and Hardness workshop'' (February
18--21, 2020, organized by Daniele Micciancio, Daniel Dadush, Chris
Peikert) which lead to the results of this paper.

\smallskip

We thank the anonymous referees for their helpful comments,
suggestions, and corrections on the manuscript.

\smallskip

This project has received funding from the European Union's Horizon
2020 research and innovation programme under the Marie
Sk\l{}odowska-Curie agreement No 764759. F.V. is partially supported
by the SFB/TRR 191 ``Symplectic Structures in Geometry, Algebra and
Dynamics'', F.V. and M.C.Z. are partially supported ``Spectral bounds
in extremal discrete geometry'' (project number 414898050), both
funded by the DFG. A.H. is partially supported by the DFG under the
Priority Program CoSIP (project number SPP1798).  A.T. is partially
funded through HYPATIA.SCIENCE by the financial fund for the
implementation of the statutory equality mandate and the Department of
Mathematics and Computer Science, University of Cologne.

\end{document}